\documentclass[11pt]{article}
\usepackage{amssymb,amsmath,amsfonts,amsthm,epsfig,latexsym,color}
\usepackage{amsfonts,amsmath,amsthm, amssymb, amscd, mathrsfs}
\usepackage{latexsym, euscript, epic, eepic}
\usepackage{graphicx}
\usepackage{verbatim}
\usepackage{fourier}
\usepackage[all]{xy}
\usepackage[colorlinks=true,backref=page]{hyperref}

\makeatletter
\newcommand{\subsectionruninhead}{\@startsection{subsection}{2}{0mm}
{-\baselineskip}{-0mm}{\bf\large}}
\newcommand{\subsubsectionruninhead}{\@startsection{subsubsection}{3}{0mm}
{-\baselineskip}{-0mm}{\bf\normalsize}}
\makeatother

\newtheorem*{theorem*}{Theorem}

\newtheorem{theoremalph}{Theorem}
\newtheorem{corollaryalph}[theoremalph]{Corollary}

\newtheorem*{proposition*}{Proposition}
\newtheorem*{corollary*}{Corollary}
\newtheorem*{claim*}{Claim}
\newtheorem*{remark*}{Remark}
\newtheorem*{problem*}{Problem}
\newtheorem{theorem}{Theorem}[section]

\newtheorem{proposition}[theorem]{Proposition}

\newtheorem{lemma}[theorem]{Lemma}
\newtheorem{notation}[theorem]{Notation}

\newtheorem{claim}[theorem]{Claim}
\theoremstyle{definition}

\newtheorem{remark}[theorem]{Remark}

\numberwithin{equation}{section}

\def\cF{\mathcal{F}}

\newcommand{\diff}{\operatorname{Diff}}

\newcommand{\id}{\operatorname{Id}}

\newcommand{\per}{\operatorname{Per}}
\newcommand{\ph}{\operatorname{PH}}
\newcommand{\dph}{\operatorname{DPH}}
\newcommand{\deph}{\operatorname{DEPH}}

\newcommand{\CR}{\operatorname{CR}}
\newcommand{\grap}{\operatorname{Graph}}
\newcommand{\length}{\operatorname{length}}

\begin{document}

\title{$C^r$-Chain closing lemma for certain partially hyperbolic diffeomorphisms}

\author{Yi Shi\footnote{Y. Shi was partially supported by National Key R\&D Program of China (2020YFE0204200, 2021YFA1001900), NSFC (12071007, 11831001, 12090015) and the Fundamental Research Funds for the Central Universities.}  
\quad and \quad 
Xiaodong Wang\footnote{X. Wang was partially supported by National Key R\&D Program of China (2021YFA1001900), NSFC (12071285) and Innovation Program of Shanghai Municipal Education Commission (No. 2021-01-07-00-02-E00087).}}

\maketitle

\begin{abstract}
    For every $r\in\mathbb{N}_{\geq 2}\cup\{\infty\}$, we prove a $C^r$-orbit connecting  lemma for dynamically coherent and plaque expansive partially hyperbolic diffeomorphisms with 1-dimensional orientation preserving center bundle. To be precise, for such a diffeomorphism $f$, if a point $y$ is chain attainable from  $x$ through pseudo-orbits, 
    then for any neighborhood $U$ of $x$ and any neighborhood $V$ of $y$, there exist true orbits  from $U$ to $V$ by arbitrarily $C^r$-small perturbations.
    As a consequence, we prove that for $C^r$-generic diffeomorphisms in this class, periodic points are dense in the chain recurrent set, and chain transitivity implies transitivity.
\end{abstract}

\section{Introduction}

Following from Yoccoz \cite{yoccoz}, the goal of the theory of dynamical systems is to understand most of the dynamics of most systems.
Here ``most systems'' means dense, generic, open and dense dynamical systems. The various perturbation techniques have played  crucial roles in the study of most systems, for instance in the exploration of the famous $C^1$-stability conjectures.
Pugh's celebrated work: the $C^1$-closing lemma~\cite{Pugh} which realizes closing a non-wandering orbit to get periodic orbits, initiated the theory of $C^1$-perturbations.
Ma\~n\'e's ergodic closing lemma~\cite{mane-ergodic-closing} provides more information on the closing points. 
Another work by Ma\~n\'e~\cite{mane-h-point} creates homoclinic points by $C^r,r=1,2$ perturbations.
These two works~\cite{mane-ergodic-closing,mane-h-point} are indispensable   in solving the $C^1$-stability conjecture~\cite{mane-ergodic-closing,mane-stability-conjecture} and ergodic closing lemma is a milestone in both differentiable dynamical systems and ergodic theory.
An improved version of $C^1$-closing lemma is obtained by Hayashi~\cite{Hayashi} that connects two orbits which visit a same small region. Hayashi's work is now called the $C^1$-connecting lemma, see also works by Wen-Xia~\cite{WX-connecting}. More developments can be found for instance in~\cite{Arnaud,gw}. 
The strongest $C^1$-perturbation lemma is the $C^1$-chain connecting lemma proved by Bonatti-Crovisier~\cite{BC-chain-connecting}, which connects pseudo orbits under small $C^1$-perturbations.
In particular, works in~\cite{BC-chain-connecting,C-IHES} realize closing a chain recurrent orbit (or a chain recurrent set).

We recall some notions and definitions now.
Let $M$ be a $C^\infty$ closed Riemmanian manifold and denote by $\diff^r(M)$ for any $r\in\mathbb{N}\cup\{\infty\}$ the space of $C^r$-diffeomorphisms on $M$ endowed with the $C^r$-topology.

Given $f\in \diff^r(M)$ and $\delta>0$. A collection of points 
$$\big\{x_a,x_{a+1},\cdots,x_b\big\} \text{~~where~~} -\infty\leq  a<b\leq +\infty$$ 
is called a \emph{$\delta$-pseudo orbit} if it holds
$$d(x_{i+1},f(x_i))<\delta~~~\text{for every $a\leq i\leq b-1$}.$$
Given two points $x,y\in M$. One says that $y$ is \emph{attainable} from $x$ under $f$, denoted by $x\prec_f y$, if for any neighborhood $U$ of $x$ and any neighborhood $V$ of $y$, there exist $z\in U$ and $n\geq 1$ such that $f^n(z)\in V$; and $y$ is \emph{chain attainable} from $x$ under $f$, denoted by $x\dashv_f y$, if for any $\delta>0$, there exists a $\delta$-pseudo orbit $\big\{x_0,x_{1},\cdots,x_n\big\}$ from $x$ to $y$, i.e. $x_0=x$ and $x_n=y$. 
A point $x\in M$ is \emph{non-wandering} if $x\prec_f x$, and the set of non-wandering points of $f$ is denoted by $\Omega(f)$. A point $x\in M$ is \emph{chain recurrent} if $x\dashv_f x$, and the set of chain recurrent points of $f$ is denoted by $\CR(f)$. The set of periodic points is denoted by $\per(f)$.

Note that compared to non-wandering/chain recurrent points, periodic points have the strongest recurrence property. One would like to ask whether or not the weaker recurrence can be perturbed into stronger ones.
Given $f\in\diff^r(M)$. A point $p\in M$ is called \emph{$C^r$-closable}, if for any $C^r$-neighborhood $\mathcal{U}$ of $f$, there exists $g\in \mathcal{U}$ such that $p\in \per(g)$. 
For uniformly hyperbolic systems, Anosov's shadowing lemma~\cite{Anosov} implies that every chain recurrent point is $C^r$-closable (in fact without perturbation).
For a general diffeomorphism $f$, as has mentioned above, Pugh's closing lemma~\cite{Pugh} verifies that every $x\in\Omega(f)$ is $C^1$-closable and Bonatti-Crovisier's chain closing lemma~\cite{BC-chain-connecting} verifies every $x\in\CR(f)$ is $C^1$-closable. 

For dynamics beyond uniform hyperbolicity and when $r\geq 2$, things turn out to be much more complicated and delicate. 
Gutierrez \cite{Gu} built an example which  showed that the local perturbation method for proving $C^1$-closing lemma does not work in the $C^2$-topology.
Ma\~n\'e~\cite{mane-h-point} proved a $C^2$-connecting lemma which creates homoclinic points for hyperbolic periodic points with some assumption in the measure sense. Recently, a series progress has been achieved in $C^r$-closing lemma for Hamiltonian and conservative surface diffeomorphisms, see \cite{AI,EH,CPZ}.

In higher dimensions, Gan and the first author \cite{GS-closing} proved the $C^r$-closing lemma for partially hyperbolic diffeomorphisms with 1-dimensional center bundle for every $r\in\mathbb{N}_{\geq 2}\cup\{\infty\}$. Inspired by~\cite{GS-closing}, we explore $C^r$-perturbation results for certain class of partially hyperbolic diffeomorphisms.
\medskip

We say that a diffeomorphism $f\in\diff^r(M)$ is \emph{partially hyperbolic} if the tangent space $TM$ admits a $Df$-invariant continuous splitting $TM=E^s\oplus E^c\oplus E^u$ and there exists $k\geq 1$ such that for any $x\in M$, it holds
$$\|Df^k|_{E^s_x}\|<\min\big\{1,m(Df^k|_{E^c_x})\big\}\leq \max\big\{1,\|Df^k|_{E^c_x}\|\big\}<m(Df^k|_{E^u_x}).$$
Here for a  linear operator  $A$, one denotes by $\|A\|$  its norm and by $m(A)$  its conorm,  i.e.
$m(A)=\inf\big\{\|Av\|\colon \|v\|=1\big\}$.

Let $\ph^r(M), r\geq 1$ be the space consisting of all $C^r$-partially hyperbolic diffeomorphisms of $M$ endowed with the $C^r$-topology, then $\ph^r(M)\subset\diff^r(M)$ is an open subset.
We say $f\in\ph^r(M)$ is \emph{dynamically coherent} if there exist two $f$-invariant foliations $\mathcal{F}^{cs}$ and $\mathcal{F}^{cu}$ that are tangent to $E^s\oplus E^c$ and $E^c\oplus E^u$ respectively. It is clear that $\mathcal{F}^c=\mathcal{F}^{cs}\cap \mathcal{F}^{cu}$ is an $f$-invariant foliation tangent to $E^c$.
For a dynamically coherent partially hyperbolic diffeomorphism $f$, one always denotes the  partially hyperbolic splitting by $TM=E_f^s\oplus E_f^c\oplus E_f^u$ or $TM=E^s\oplus E^c\oplus E^u$ for simplicity when there is no confusion. Also, one denotes the stable/center/unstable foliations by $\mathcal{F}^s/\mathcal{F}^c/\mathcal{F}^u$.

For $f\in\ph^r(M)$ being dynamically coherent and a constant $\varepsilon>0$, 
an $\varepsilon$-pseudo orbit $\big\{x_i\big\}_{i=a}^b$ is called an \emph{$\varepsilon$-center pseudo orbit}, if $f(x_i)\in\mathcal{F}^c_{\varepsilon}(x_{i+1})$ for any $a\leq i\leq b-1$.
We say $f$ is \emph{$\varepsilon$-plaque expansive} if any two bi-infinite $\varepsilon$-center pseudo orbit $\big\{x_i\big\}_{i\in\mathbb{Z}}, \big\{y_i\big\}_{i\in\mathbb{Z}}$ satisfying 
$$d(x_i,y_i)<\varepsilon \text{~~for every $i\in\mathbb{Z}$,}$$
must satisfy $y_i\in\mathcal{F}^c_{\varepsilon}(x_i)$
for every $i\in\mathbb{Z}$. Here $\mathcal{F}^c_{\varepsilon}(x)$ is the disk  centered at $x$ in $\cF^c(x)$ with radius $\varepsilon$. We say $f$ is \emph{plaque expansive} if it is $\varepsilon$-plaque expansive for some $\varepsilon>0$.

\begin{notation}\label{Notation}
	For $r\in\mathbb{N}_{\geq 2}\cup\{\infty\}$, we denote by $\dph_1^r(M)$ {\rm(resp. $\deph_1^r(M)$)} the set of all partially hyperbolic diffeomorphisms $f\in\ph^r(M)$ that satisfies properties {\rm (a) \& (b) (resp. (a), (b) \& (c))}:
	\begin{itemize} 
		\item[\rm (a)] the center bundle $E^c$ is orientable with ${\rm dim}E^c=1$, and $Df$ preserves the orientation;
		\item[\rm (b)] $f$ is dynamically coherent; 
		
		\item[\rm (c)] $f$ is plaque expansive.
	\end{itemize}
\end{notation}

It has been proved \cite[Theorem 7.1\&7.4]{hps} that dynamically coherence plus plaque expansiveness is a robust property for $C^1$-partially hyperbolic diffeomorphisms. Thus $\deph_1^r(M)$ forms an open set in $\diff^r(M)$ for every $r\geq1$. There are a series of partially hyperbolic diffeomorphisms belongs to $\deph_1^r(M)$:
\begin{enumerate}
	\item derived-from-Anosov partially hyperbolic diffeomorphisms on 3-torus $\mathbb{T}^3$ \cite{H1,Po}, where $Df$ preserves the center orientation;
	\item partially hyperbolic diffeomorphisms with circle center fibers on 3-torus $\mathbb{T}^3$ \cite{H1,Po} and 3-nilmanifolds $N$ \cite{HP}, where $Df$ preserves the center orientation;
	\item discretized Anosov flows \cite{Martinchich}, i. e. partially hyperbolic diffeomorphisms which are leaf conjugate to  Anosov flows.
\end{enumerate}
In particular, every dynamically coherent and plaque expansive partially hyperbolic diffeomorphism with 1-dimensional center admits a double cover with $E^c$ orientable and an iteration of the lifting diffeomorphism belongs to $\deph_1^r(M)$.

\begin{remark}
	It has been conjectured \cite{hps} that every dynamically coherent partially hyperbolic diffeomorphism must be plaque expansive. Actually, all known examples of dynamically coherent partially hyperbolic diffeomorphisms are plaque expansive.
\end{remark}

\bigskip


%

In this paper, we prove that for every $r\in\mathbb{N}_{\geq 2}\cup\{\infty\}$ and every $f\in\dph_1^r(M)$, if a point $y$ is chain attainable from  $x$, then any small neighborhoods of $x$ and $y$ can be connected through true orbits by arbitrarily small $C^r$-perturbations of $f$.
Let $\mathcal{X}^r(M)$ be the space of $C^r$-vector fields on $M$ for $r\in\mathbb{N}\cup\{\infty\}$. Given  $X\in\mathcal{X}^r(M)$ and $\tau\in\mathbb{R}$, denote by $X_\tau$  the time-$\tau$ map generated by $X$. We say  $X$ is transverse to a co-dimension one bundle $E^s\oplus E^u$ if $X(x)$ is transverse to $E^s_x\oplus E^u_x$ in $T_xM$ for every $x\in M$.

\begin{theoremalph}\label{Thm:main}
	For any $r\in\mathbb{N}_{\geq 2}\cup\{\infty\}$, assume $f\in\dph_1^r(M)$ and  $X\in\mathcal{X}^r(M)$ is a vector field that is transverse to $E^s\oplus E^u$.
	For any two points $x,y\in M$ with $x\dashv_f y$, there exist $p_k\in M$, $\tau_k\in \mathbb{R}$ and $n_k\in \mathbb{N}$ such that
	\begin{itemize}
		\item $\tau_k\rightarrow 0$ as $k\rightarrow\infty$; 
		\item $p_k\rightarrow x$ and $(X_{\tau_k}\circ f)^{n_k}(p_k)\rightarrow y$ as $k\rightarrow \infty$.
	\end{itemize} 
	In particular, if $y=x$, then $(X_{\tau_k}\circ f)^{n_k}(p_k)=p_k$. 
	This means every chain recurrent point of $f\in\dph_1^r(M)$ is $C^r$-closable.
\end{theoremalph}

\begin{remark}
	When $r=1$, this orbit connecting lemma was proved by Bonatti and Crovisier in \cite{BC-chain-connecting} for every $f\in\diff^1(M)$. See \cite{C-IHES,C-habitation} for more developments and applications of chain connecting lemmas in $C^1$-topology.
\end{remark}

%

Recall that $f\in\diff^1(M)$ is \emph{transitive} if for any two open sets $U,V\subset M$, there exists $n\geq 1$ such that $f^n(U)\cap V\neq\emptyset$. This is equivalent to $x\prec_f y$ for every $x,y\in M$. Similarly, $f$ is chain transitive if $x\dashv_f y$ for every pair of points $x,y\in M$. It is clearly that transitivity implies chain transitivity. 
Concerning that $\deph_1^r(M)$ forms an open set in $\diff^r(M)$ for every $r\geq1$, we could study generic properties for $\deph_1^r(M)$.
We show that for  $C^r$-generic $f\in\deph_1^r(M)$, transitivity and chain transitivity are equivalent.
Recall that a property is \emph{generic} if it holds for every element in a dense $G_\delta$ subset.

\begin{corollaryalph}\label{Thm:generic}
	For every $r\in\mathbb{N}_{\geq 2}\cup\{\infty\}$, there exists a dense $G_\delta$ subset  $\mathcal{R}\subset \deph_1^r(M)$ such that for any $f\in\mathcal{R}$, the following properties hold.
	\begin{enumerate}
		\item\label{Item:equivalent-of-attainability} $x\dashv_f y$ if and only if $x\prec_f y$.
		\item\label{Item:dense-periodic orbit} Periodic points are dense in the chain recurrent set: $\CR(f)=\Omega(f)=\overline{\per(f)}$.
		\item\label{Item:global-transitive} If $f$ is chain transitive on $M$, then $f$ is transitive on $M$.
	\end{enumerate}
\end{corollaryalph}

\begin{remark}
	Transitive dynamics that are beyond uniform hyperbolicity have attracted much interest from dynaminists, for instance~\cite{BD-robust-transitive-diffeo,Br}.
	There exist chain transitive partially hyperbolic diffeomorphisms with 1-dimensional center bundle which are accessible but not transitive \cite{GS-pams}. It is worthy to see whether Corollary \ref{Thm:generic} holds for $C^r$-generic partially hyperbolic diffeomorphisms with 1-dimensional center bundle.
\end{remark}

\subsection*{Acknowledgments}
We are very grateful to the anonymous referees for their valuable suggestions.

\section{Lifting dynamics and leaf conjugacy}\label{Section:preparations}
As have been mentioned, our proof is inspired by~\cite{GS-closing}.
This section devotes to state the necessary preliminaries, some of which are borrowed from~\cite{GS-closing}. We first give an outline.

\begin{itemize}
	\item Firstly, we give a Lipschitz center shadowing property for dynamically coherent partially hyperbolic diffeomorphisms proved in~\cite{KT-center shadowing} which insures that one can always consider center pseudo orbits in our setting.
	
	\item Secondly, we parameterize each center leaf through the real line $\mathbb{R}$ and associate it with an order from $\mathbb{R}$ by the fact that $\dim(E^c)=1$. Then one could lift the center dynamics to $\bigsqcup\limits_{i\in\mathbb{Z}}\mathbb{R}_i$ where each $\mathbb{R}_i=\mathbb{R}$. Moreover, we prove that by replacing some intermediate points of a center pseudo orbit if necessary, one could always assume the center pseudo orbit jumps in an ordered way.  This will be done in Section~\ref{Section:lifting-dynamics}.
	
	\item Thirdly, we extend the lifting dynamics on $\bigsqcup\limits_{i\in\mathbb{Z}}\mathbb{R}_i$ to a tubular neighborhood by pulling back a $C^\infty$-bundle $F^s\oplus F^u$ which is $C^0$-close to $E^s\oplus E^u$. Theorem~\ref{Thm:Lip-shadowing} guarantees that the center curves for the dynamics in the tubular neighborhood admits a Lipschitz shadowing property under perturbations. This appears in Section~\ref{Seciton:lifting-bundle} and~\ref{Section:Lipschtz-shadowing}.
	
	\item Finally, based on the preparations in this section, we will prove a push-forward (or push-backward) perturbation result in Section~\ref{Section:proof-main} which allows us to prove Theorem~\ref{Thm:main} by more delicate estimates.
\end{itemize}


One first recalls the following proposition from~\cite[Theorem 1]{KT-center shadowing} which states a Lipschitz center shadowing property for dynamically coherent partially hyperbolic diffeomorphisms. Roughly speaking, for such a system, any pseudo orbit with small jumps would be shadowed by a center pseudo orbit.
See also similar results in~\cite{Hu-Zhou-Zhu}.
\begin{proposition}[\cite{KT-center shadowing}]\label{Pro:center-pseudo-orbit}
	Assume $f\in\ph^1(M)$ is dynamically coherent.
	Then there exist $\varepsilon_0>0$ and $L>1$ such that for any $\varepsilon\in (0,\varepsilon_0)$ and any $\varepsilon$-pseudo orbit $\big\{w_a,w_{a+1},\cdots,w_b\big\}, -\infty\leq a<b\leq +\infty$, there exists an $L\varepsilon$-center pseudo orbit $\big\{x_a,x_{a+1},\cdots,x_b\big\}$ satisfying 
	$$d(x_i,w_i)<L\varepsilon ~~\text{for any $a\leq i\leq b$}.$$
    Moreover, if 
     $w_a=w_b$ where $-\infty<a<b<+\infty$, then the new center pseudo orbit can be chosen to satisfy $x_a=x_b$.
\end{proposition}

\begin{proof}[Proof of the ``moreover'' part]
	
	The statement of Proposition~\ref{Pro:center-pseudo-orbit} before the ``moreover'' part is~\cite[Theorem 1]{KT-center shadowing}. Thus we only need to prove the ``moreover'' part.
	We first revisit the steps in the proof of~\cite[Theorem 1]{KT-center shadowing}. Note that one could always consider a bi-infinite pseudo orbit $\{w_i\}_{i\in\mathbb{Z}}$ because any finite or one-sided-infinite pseudo orbit can be extended naturally to a bi-infinite one (see for instance Section~\ref{Section:lifting-dynamics} below).
	\begin{itemize}
		\item [1.] Fix two constants $\delta_0>0, L_0>1$ such that for any $0<\delta\le\delta_0$ and any two points $x,y\in M$ satisfying $d(x,y)<\delta$, then  $\mathcal{F}_x^{cu}(L_0\delta)$ intersects $\mathcal{F}_y^s(L_0\delta)$ at a unique point and similarly for $\mathcal{F}_x^{cs}(L_0\delta)$ and $\mathcal{F}_y^u(L_0\delta)$. 
		This local product property is guaranteed by the transversality and continuation of local foliations.
		
		\item [2.] By considering an iterate of $f$~\cite[Lemma 1]{KT-center shadowing} (or by taking an adapted metric as in Section~\ref{Section:preparations}) and by decreasing $\delta_0$ if necessary, one has 
		\[d^s(f(x),f(y))<\lambda d^s(x,y), \forall y\in\mathcal{F}^s_x(\delta_0),\forall x\in M,\]
		\[d^u(f^{-1}(x),f^{-1}(y))<\lambda d^s(x,y), \forall y\in\mathcal{F}^u_x(\delta_0),\forall x\in M,\]
		where $d^{s/u}$ denotes the metric on the stable/unstable foliations and $\lambda\in(0,1)$ is a constant depending only on the contracting/expanding ratio of $Df$ along $E^s/E^u$.
		
		\item[3.] Then there exist two constants $\varepsilon_0\in(0,\delta_0)$ and $L_1>L_0$ such that for any $\varepsilon\in(0,\varepsilon_0)$ and  any bi-infinite $\varepsilon$-pseudo orbit $\{w_i\}_{i\in\mathbb{Z}}$, the  map $h_i^s\colon \mathcal{F}^s_{w_i}(L_1\varepsilon)\rightarrow \mathcal{F}^s_{w_{i+1}}(L_1\varepsilon)$ is well defined where for every $z\in\mathcal{F}^s_{w_i}(L_1\varepsilon)$, the image $h_i^s(z)$   is the unique point of $\mathcal{F}^{cu}_{f(z)}(L_0\delta_0)\cap \mathcal{F}^s_{w_{i+1}}(L_1\varepsilon)$. This corresponds to~\cite[Lemma 2]{KT-center shadowing}.

		\item[4.] Consider the set $X^s=\prod_{i=-\infty}^{+\infty}\mathcal{F}^s_{w_i}(L_1\varepsilon)$ endowed with the Tikhonov product topology and consider the map $H\colon X^s\rightarrow X^s$ with 
		\[H(\{z_i\})=\{z'_{i+1}\},~~\text{where~~} z'_{i+1}=h^s_i(z_i).\footnote{A bit difference with~\cite{KT-center shadowing} is that they are working on the tangent bundle while here we consider local leaves on the manifold directly. These two ways are equivalent through the exponential map.}\]
		Applying the Tikhonov-Schauder fixed point theorem, the map $H$ has a (maybe not unique) fixed point. This induces an orbit  of center unstable leaves $\big\{\mathcal{F}^{cu}_{y^s_i}\big\}_{i\in\mathbb{Z}}$ $\big($i.e. $f\big(\mathcal{F}^{cu}_{y^s_i}\big)=\mathcal{F}^{cu}_{y^s_{i+1}}$ for any  $i\in\mathbb{Z}$ $\big)$ where $y^s_{i+1}\in \mathcal{F}^{cu}_{f(y^s_i)}(L_1\varepsilon)$ and
		which satisfies that $y^s_{i}$ is close to $w_i,\forall i\in\mathbb{Z}$.
		Symmetrically, one obtains an orbit  of center stable leaves $\big\{\mathcal{F}^{cs}_{y^u_i}\big\}_{i\in\mathbb{Z}}$ which satisfies that $y^u_{i}$ is close to $w_i,\forall i\in\mathbb{Z}$. 
		
		\item[5.] The intersection $\big\{\mathcal{F}^{cs}_{y^u_i}\cap\mathcal{F}^{cu}_{y^s_i}\big\}_{i\in\mathbb{Z}}$ forms an  orbit of center leaves $\big\{\mathcal{F}^{c}_{y_i}\big\}_{i\in\mathbb{Z}}$ with $y_i$ being close to $w_i$ for all $i\in\mathbb{Z}$. Finally, following arguments in~\cite[Page 2907-2908]{KT-center shadowing}, there exists a  center pseudo orbit $\big\{x_i\big\}_{i\in\mathbb{Z}}$  that shadows $\big\{w_i\big\}_{i\in\mathbb{Z}}$. In this step, the constant $\varepsilon_0$ maybe decreasing once again and $L_1$ is replaced by a larger $L$.
	\end{itemize}

	The ``moreover'' part is almost included in the proof of~\cite[Theorem 1]{KT-center shadowing}. We give an explanation here.  To simplify symbols, we let $0=a<b=n<+\infty$ and consider a  pseudo orbit $\big\{w_0,w_{1},\cdots,w_n\big\}$ with $w_n=w_0$. One first repeats it forward and backward to get a bi-infinite  periodic pseudo orbit, i.e. let $w_i=w_{i~{\rm mod}~n}$ for $i\in\mathbb{Z}$. Then Steps 1,2,3 still work. We only need to verify that in Step 4, one could choose a fixed point  of $H$ which induces a periodic orbit of center unstable leaves  $\big\{\mathcal{F}^{cu}_{y^s_i}\big\}_{i\in\mathbb{Z}}$, i.e. $y^s_i=y^s_{i~{\rm mod}~n}$. 
	To do this, we  consider the finite composition (rather than considering $H$)
	\[ H^s_n= h^s_{n-1} \cdots h^s_1\circ h^s_0\colon \mathcal{F}^s_{w_0}(L_1\varepsilon)\rightarrow \mathcal{F}^s_{w_{n}}(L_1\varepsilon)=\mathcal{F}^s_{w_0}(L_1\varepsilon),\]
	which is well defined by Step 3.
	Note that each $h^s_i$ is a composition of a local $cu$-holonomy with $f$, thus $H^s_n$ is continuous. On the other hand, $H^s_n$ maps $\mathcal{F}^s_{w_0}(L_1\varepsilon)$ into itself where $\mathcal{F}^s_{w_0}(L_1\varepsilon)$ is an $s$-dimensional disk.
	Applying the Tikhonov-Schauder fixed point theorem, the map $H^s_n$ has a (maybe not unique) fixed point $y^s_0\in \mathcal{F}^s_{w_0}(L_1\varepsilon)$. This induces an orbit segment of center unstable leaves  $\big\{\mathcal{F}^{cu}_{y^s_i}\big\}_{i\in[0,n]}$ with $f\big(\mathcal{F}^{cu}_{y^s_i}\big)=\mathcal{F}^{cu}_{y^s_{i+1}}$  for $0\leq i\leq n-1$ and $y^s_0=y^s_n$. Moreover, one also has that $y^s_{i}$ is close to $w_i,\forall i\in[0,n]$ following arguments in~\cite[Page 2907-2908]{KT-center shadowing}.
	Repeating $\big\{\mathcal{F}^{cu}_{y^s_i}\big\}_{i\in[0,n]}$ forward and backward, i.e. let $y^s_i=y^s_{i~{\rm mod}~n}$ for $i\in\mathbb{Z}$, one gets a periodic orbit of center unstable leaves 
	$\big\{\mathcal{F}^{cu}_{y^s_i}\big\}_{i\in\mathbb{Z}}$. Symmetrically we can get a periodic orbit of center stable leaves 
	$\big\{\mathcal{F}^{cs}_{y^u_i}\big\}_{i\in\mathbb{Z}}$. As in Step 5, the intersection $\big\{\mathcal{F}^{cs}_{y^u_i}\cap\mathcal{F}^{cu}_{y^s_i}\big\}_{i\in\mathbb{Z}}$ forms a periodic  orbit of center leaves $\big\{\mathcal{F}^{c}_{y_i}\big\}_{i\in\mathbb{Z}}$ with $y_i$ being close to $w_i$ for all $i\in[0,n]$ and $y_i=y_{i~{\rm mod}~n}$. We take $x_i$ to be the unique intersection of local leaf of $\mathcal{F}^{cu}_{y^s_i}$  with the local leaf of $\mathcal{F}^{s}_{y^u_i}$.  Here $\{x_i\}_{i=0}^n$ is the periodic $L\epsilon$-center pseudo orbit we need, see \cite[Page 2908]{KT-center shadowing} for the estimation of the constant $L$ \footnote{This is the first line of ~\cite[Page 2908]{KT-center shadowing} where they use the notation ``$y_k$'' while here we use ``$x_i$''.}.
\end{proof}

\begin{remark}\label{Rem:center-pseudo-orbit}	
	As a consequence of Proposition~\ref{Pro:center-pseudo-orbit}, for a  dynamically coherent $f\in\ph^1(M)$, for any $x\in\CR(f)$ and any $\varepsilon>0$, there exists a periodic $\varepsilon$-center pseudo orbit 
	\[\big\{x_0,x_1,\cdots,x_n=x_0\big\} \text{~~satisfying~~}	d(x_0,x)<\varepsilon.\]	
\end{remark}

\subsection{Preliminary settings}\label{Section:setting}

From now on, unless emphasized, we always assume $f\in\dph_1^r(M), r\in\mathbb{N}_{\geq 2}\cup\{\infty\}$. 
In this subsection, we will fix some elements for $f\in\dph_1^r(M), r\in\mathbb{N}_{\geq 2}\cup\{\infty\}$. The first is a metric on $M$ that is adapted to the partially hyperbolic splitting. Secondly, one associates an order to each ($1$-dimensional) center bundle. Finally, one fixes two $C^\infty$-sub-bundles $F^s$ and $F^u$ that are $C^0$-close to $E^s$ and $E^u$ respectively.

Let $d$ be a $C^{\infty}$-Riemannian metric on $M$ that is adapted~\cite{gourmelon} to  the splitting $E^s\oplus E^c\oplus E^u$. To be precise, there exist  constants $0<\lambda<1$ and $0<\eta_0<10^{-3}$ such that for any $z\in M$, it satisfies:
\begin{itemize}
	\item $\|Df|_{E^s_z}\|<\min\big\{\lambda,\|Df|_{E^c_z}\|\big\}\leq \max\big\{\lambda^{-1},\|Df|_{E^c_z}\|\big\}<m(Df|_{E^u_z})$;
	\item the three sub-bundles are almost orthogonal mutually:
	  $$\min\limits_{v^{\alpha}\in E^{\alpha}_z,\|v^{\alpha}\|=1} d_{T_zM}(v^{\alpha},E^{\beta}_z\oplus E^{\gamma}_z)>1-\eta_0, \text{~~where~~} \big\{\alpha,\beta,\gamma\big\}=\big\{s,c,u\big\};$$
	\item for $\alpha=s,c,u$, the projection $\pi_z^{\alpha}\colon T_zM\rightarrow E^{\alpha}_z$ satisfies $\|\pi_z^{\alpha}(v)\|\leq 2\|v\|$ for any $v\in T_zM$.
\end{itemize}
Here $\|\cdot\|$ is the vector norm on $TM$ and $d_{T_zM}(\cdot,\cdot)$ is the distance induced by the Riemannian metric $d$. Such an adapted Riemmanian metric $d$ always exists by the classical arguments in~\cite{hps}.

Denote by $\pi\colon TM\rightarrow M$ the natural projection from $TM$ to $M$ and by $\pi^\alpha \colon TM\rightarrow E^\alpha$ the bundle projection map. For $v\in TM$, let $v=v^s+v^c+v^u$ be the direct sum splitting with $v^\alpha=\pi^\alpha(v)$, $\alpha=s,c,u$. 
To simplify notations, denote $s=\dim(E^s)$ and $u=\dim(E^u)$. 
For any constant $a>0$, the $a$-cone field $\mathcal{C}_a(E^c)$ of $E^c$ is defined as
$$\mathcal{C}_a(E^c)=\left\{v\in TM \colon \|v^s+v^u\|\leq a\|v^c\|\right\}.$$
A $C^1$-curve $\gamma\colon I\rightarrow M$, where $I$ is an interval, is called tangent to the $a$-cone field of $E^c$ if the tangent line of $\gamma$ is contained in $\mathcal{C}_a(E^c)$ everywhere.

Let $X\in\mathcal{X}^r(M)$ be a vector field that is transverse to $E^s\oplus E^u$ everywhere.
Since $\dim(E^c)=1$, one has that $E^c_z\cong \mathbb{R}$, for any $z\in M$. Moreover, since $E^c$ is orientable and $Df$ preserves the orientation,  one can associate to each $E^c_z$ an order induced by $\mathbb{R}$ such that $Df|_{E^c}$ keeps this order.
Without loss of generality, we assume that $X\in\mathcal{X}^r(M)$  is  positively transverse to $E^s\oplus E^u$ everywhere. To be precise,  it satisfies that $\pi^c(X(z))>0$ with respect to the order of $E^c_z$ for every $z\in M$. The other case when $X$ is negatively transverse to $E^s\oplus E^u$ everywhere is symmetric by taking the inverse orientation of $E^c$. 

Note that in general $E^s$ and $E^u$ are only continuous sub-bundles of $TM$. One takes two $C^\infty$-sub-bundles $F^s$ and $F^u$ of $TM$ that are $\eta_0$-$C^0$-close to $E^s$ and $E^u$ respectively. To be precise,
$$\dim(F^\alpha)=\dim(E^\alpha)
\quad\text{and}\quad
\angle(F^\alpha,E^\alpha)=\max\limits_{z\in M}\max\limits_{v\in F^\alpha_z,\|v\|=1}d_{T_zM}(v,E^\alpha_z)<\eta_0, \quad\text{for}~ \alpha=s,u.$$
Thus $TM=F^s\oplus E^c\oplus F^u$ is still a direct sum decomposition.
In particular, by taking $F^s$ and $F^u$ close enough to $E^s$ and $E^u$ respectively, one can guarantee that $X$ is also transverse to $F^s\oplus F^u$.

\subsection{Parameterizing and lifting center dynamics}\label{Section:lifting-dynamics}
For each $z\in M$, denote by $\mathcal{F}^c_z$  the  center leaf of $z$. 
Based on the fact that $\dim(\mathcal{F}^c_z)=1$, one  parameterizes the center leaves $\big\{\mathcal{F}^c_z\big\}_{z\in M}$ and lift the center dynamics on the parameterization.
Then we prove that every center pseudo orbit   with small jumps could be replaced by one that jumps in an ordered way and that preserves the initial and end points, see Lemma~\ref{Lem:ordered-center-chain}.

\paragraph{Parameterization of the center foliation} 

Note that $\mathcal{F}^c_z$ either is non-compact and diffeomorphic to $\mathbb{R}$, or is compact and diffeomorphic to $\mathbb{S}^1$. Let $\theta_z\colon \mathbb{R}\rightarrow \mathcal{F}^c_z$ be a $C^1$ map such that 
\begin{itemize}
	\item $\theta_z(0)=z$;
	\item $\theta_z'(t)$ is the positive unit vector of $E^c_{\theta_z(t)}$, for all $t\in\mathbb{R}$.
\end{itemize}
To emphasize the dependence of the base point $z$, we denote by $\mathbb{R}_z=\theta_z^{-1}(\mathcal{F}^c_z)$.
Note that when $\mathcal{F}^c_z$ is non-compact, then $\theta_z$ is a diffeomorphism and is isometric in the sense that $d^c(\theta_z(t_1),\theta_z(t_2))=|t_2-t_1|$ where $d^c(\cdot,\cdot)$ is the metric on the center foliations. 
When $\mathcal{F}^c_z$ is compact, then $\theta_z$ is a universal covering map and is locally isometric.

\paragraph{The constant $\varepsilon_0$}
Since the length of all (compact) center leaves has a uniform lower bound, one can take a small constant $0<\varepsilon_0<1$ such that for every $z\in M$ and every interval $I$ with length no larger than $10\varepsilon_0$, the restricted map $\theta_z|_{I}$ is isometric. 



\medskip

Take $\varepsilon\in(0,\varepsilon_0)$. Assume 
$$\Gamma=\big\{x_0,x_1,\cdots,x_n\big\}$$ 
is an $\varepsilon$-center pseudo orbit. 
One extends $\Gamma$ to be bi-infinite. To be precise, let
$$\displaystyle\Gamma_\infty=\big\{x_i\big\}_{i=-\infty}^{+\infty}=\big\{\cdots,x_{-1},x_0,x_1,\cdots,x_n,x_{n+1},\cdots\big\}$$ 
such that 
$$x_i=f^i(x_0) \text{~when~} i\leq -1 
\quad\text{and}\quad
 x_i=f^{i-n}(x_n) \text{~when~} i\geq n+1.$$
To simplify notations, denote $\mathbb{R}_{i}=\mathbb{R}_{x_i}$ and $\theta_i=\theta_{x_i}\colon \mathbb{R}_{i}\rightarrow \mathcal{F}^c_{x_i}$. 
Since $f(x_i)\in \mathcal{F}^c_{x_{i+1}}$ and each $\theta_i$ is locally isometric, then there exists a unique $t_{i+1}\in(-\varepsilon,\varepsilon)$ such that $f(x_i)=\theta_{i+1}(t_{i+1})$. In particular, one has $t_{i+1}=0$ when $i\leq -1$ or $i\geq n$.

\paragraph{Lifting dynamics} Note that $\big\{\theta_z\big\}_{z\in M}$ defines a line-fiber over $M$. 
The dynamical system  $f\colon M\rightarrow M$ induces a natural dynamics  $\displaystyle \widehat{f}\colon \bigsqcup\limits_{i\in\mathbb{Z}}\mathbb{R}_i\rightarrow \bigsqcup\limits_{i\in\mathbb{Z}}\mathbb{R}_i$ such that for $\forall i\in\mathbb{Z}$, the restriction $\widehat{f}|_{\mathbb{R}_i}\colon \mathbb{R}_i\rightarrow \mathbb{R}_{i+1}$ is a diffeomorphism and   for $\forall t\in\mathbb{R}_i,$ it satisfies
	$$\theta_{i+1}\circ \widehat{f}(t)=f\circ\theta_i(t)~~~~~\text{and }~~~|\widehat{f}(0_i)|<\varepsilon.$$
That is to say the following diagram commutes:
\begin{displaymath}
	\xymatrix{ \bigsqcup\limits_{i\in\mathbb{Z}}\mathbb{R}_i \ar[r]^{\widehat{f}} \ar[d]_{\theta} & \bigsqcup\limits_{i\in\mathbb{Z}}\mathbb{R}_{i} \ar[d]^{\theta} \\
			M  \ar[r]^f& M}
\end{displaymath}
where $\theta|_{\mathbb{R}_i}=\theta_i$ for $i\in\mathbb{Z}$.
Since $Df$ preserves the orientation of $E^c$, each $\widehat{f}|_{\mathbb{R}_z}$ is strictly increasing.
In particular, one has that $\widehat f(0_i)=t_{i+1}$ for every $i\in\mathbb{Z}$.
Moreover, the map $\theta\colon \bigsqcup\limits_{i\in\mathbb{Z}}\mathbb{R}_i\rightarrow M$ is a normally hyperbolic leaf immersion following~\cite[Page 69]{hps}, see also~\cite[Proposition 4.4]{GS-closing}.

\begin{remark}
	When $\mathcal{F}^c_{x_0}$ is non-compact, so as $\mathcal{F}^c_{x_i}$ for all $i\in\mathbb{Z}$, then $\theta_i$  is a diffeomorphism, and $\widehat{f}|_{\mathbb{R}_i}\colon\mathbb{R}_i\rightarrow \mathbb{R}_{i+1}$ is defined as $\widehat{f}|_{\mathbb{R}_i}=\theta_{i+1}^{-1}\circ f\circ\theta_i$. When otherwise $\mathcal{F}^c_{x_0}$  is compact, so as $\mathcal{F}^c_{x_i}$ for all $i\in\mathbb{Z}$, then each $\theta_i$ can be seen as a universal covering map of $\mathbb{S}^1$ and $f\colon \mathcal{F}^c_{x_i}\rightarrow \mathcal{F}^c_{x_{i+1}}$ can be seen as a circle diffeomorphism, thus one can take $\widehat{f}|_{\mathbb{R}_i}\colon\mathbb{R}_i\rightarrow \mathbb{R}_{i+1}$ to be the unique lifting map of $f|_{\mathcal{F}^c_{x_i}}$ that satisfies $|\widehat{f}(0_i)|<\varepsilon$. Moreover, in both cases, $\widehat f(0_i)=0_{i+1}$ whenever $i\leq -1$ or $i\geq n$.
\end{remark}

\begin{remark}
	The reason why we extend the center pseudo-orbit $\Gamma$ to the bi-infinite one $\Gamma_\infty$ is that in this way one obtains an invariant set $\bigsqcup\limits_{i\in\mathbb{Z}}\mathbb{R}_i$ of the lifting  system $\widehat{f}$. On the other hand, one could  lift the center dynamics as above over a general bi-infinite pseudo-orbit, but we will concentrate on  the finite part $\Gamma$ in the following.
\end{remark}

Recall that $f(x_i)=\theta_{i+1}(t_{i+1})$ which is equivalent to say $\widehat f(0_i)=t_{i+1}$, with $|t_{i+1}|<\varepsilon, \forall i\in\mathbb{Z}$ and in particular $t_{i+1}=0$ when $i\leq -1$ or $i\geq n$.
The following lemma states that replacing some points if necessary, the center pseudo orbit $\Gamma$ can be  chosen to jump in an ordered way.

\begin{lemma}\label{Lem:ordered-center-chain}
	Replacing some of the intermediate points in $\big\{x_1,x_2\cdots,x_{n-1}\big\}$ if necessary, one can assume that the $\varepsilon$-center pseudo orbit $\Gamma$ from $x_0$ to $x_n$  satisfies that all $t_i, 1\leq i\leq n$ have the same sigh. That is to say either $-\varepsilon<t_i\leq 0$ for all $1\leq i\leq n$ or $0\leq t_i<\varepsilon$ for all $1\leq i\leq n$. 
\end{lemma}

\begin{proof}
	Note that $f(\mathcal{F}^c_{x_i})=\mathcal{F}^c_{x_{i+1}}$.
	We split the proof into two cases whether the center leaves are compact or not.
	
	\paragraph{Non-compact case} Assume each $\mathcal{F}^c_{x_i}, i=0,1,\cdots,n$ is non-compact and thus diffeomorphic to $\mathbb{R}$ by $\theta_i$. To simplify notations, we define an order ``$\leq$'' on each $\mathcal{F}^c_{x_i}$ inherited from that of $\mathbb{R}$. To be precise, for every $w,w'\in \mathcal{F}^c_{x_i}$, one defines $w\leq w'$ (resp. $w< w'$) if and only if $\theta_i^{-1}(w)\leq \theta_i^{-1}(w')$ (resp. $\theta_i^{-1}(w)< \theta_i^{-1}(w')$).
	
	Notice that $f^i(x_0)\in\mathcal{F}^c_{x_i}$ for each $i=0,1,\cdots,n$. 
	If in the trivial case $f^n(x_0)= x_{n}$, then one just takes $x_i=f^i(x_0)$ for each $i=1,2,\cdots n-1$ and as a result  every $t_i=0$.
	So in the following, without loss of generality,  assume that $f^n(x_0)< x_{n}$ and the other case where $x_n<f^n(x_0)$ is symmetric. 
	One concludes Lemma~\ref{Lem:ordered-center-chain} in the non-compact case from the following claim.
	
	\begin{claim}\label{Claim:keep-order-noncompact}
		Replacing some of the intermediate points in $\big\{x_1,x_2,\cdots,x_{n-1}\big\}$ if necessary, we can choose $\Gamma$ to satisfy that  $f^{i+1}(x_0)\leq f(x_i)\leq x_{i+1}$  for any $0\leq i\leq n-1$. This implies that $-\varepsilon<t_i\leq 0$ for all $1\leq i\leq n$. 
	\end{claim}
	\begin{proof}[Proof of Claim~\ref{Claim:keep-order-noncompact}]
		By the fact that $Df$ preserves the orientation of $\mathcal{F}^c$, the assumption $f^n(x_0)< x_{n}$ implies $f^{i}(x_0)< f^{-n+i}(x_n)$ for any $i=0,1,\cdots, n$. 
		\medskip
		
		We associate to the center pseudo orbit $\Gamma=\big\{x_0,x_1,\cdots,x_n\big\}$ an integer
		$k=k(\Gamma)\in [0,n]$ defined by the following:
		\begin{itemize}
			\item If $f^{n}(x_0)\leq f(x_{n-1})\leq x_{n}$ holds, then $k$ is the smallest integer in $[0,n-1]$ such that for any $i\in [k,n-1]$, it satisfies $f^{i+1}(x_0)\leq f(x_i)\leq x_{i+1}$;
			\item If otherwise, take $k=n$.
		\end{itemize} 
	
		If $k=k(\Gamma)=0$, then Claim~\ref{Claim:keep-order-noncompact} satisfies automatically for $\Gamma$. Thus In the following one assumes that $k=k(\Gamma)\geq 1$.
		
		By the choice of $k$, one knows that the inequality $f^{k}(x_0)\leq f(x_{k-1})\leq x_{k}$ does not hold. Note that if $k$ is taken as in the first case, then $f^{k+1}(x_0)\leq f(x_k)$ implies $f^{k}(x_0)\leq x_k$; and if in the second case where $k=n$, then one also has $f^{k}(x_0)\leq x_k$ by assumption.
		Thus considering the orders of $f^{k}(x_0), f(x_{k-1})$ and $x_{k}$, 
		there are the following two cases: either $f(x_{k-1})<f^k(x_0)\leq x_k$ or $f^k(x_0)\leq x_k<f(x_{k-1})$.
		\medskip
		
		\noindent\textit{Sub-case (a):} If $f(x_{k-1})<f^k(x_0)\leq x_k$,  the  fact $d^c(f(x_{k-1}),x_k)<\varepsilon$  implies $d^c(f^k(x_{0}),x_k)<\varepsilon$. Thus the new collection  
		$$\big\{x_0,f(x_0),\cdots,f^{k-1}(x_0),x_k,x_{k+1},\cdots,x_{n-1},x_n\big\}$$ 
		is an $\varepsilon$-center pseudo orbit that satisfies Claim~\ref{Claim:keep-order-noncompact} and one just denotes it by $\Gamma'$.
		\bigskip
		
		\noindent\textit{Sub-case (b):} Now we assume $f^k(x_0)\leq x_k<f(x_{k-1})$.
		By the orientation preserving property, we have $f^{-1}(x_k)<x_{k-1}$ and $x_0\leq f^{-k}(x_k)$.
		Thus there exists a minimal $j\in[1,k-1]$ such that $x_i\leq f^{-k+i}(x_k)$ for any $i\in[0,j-1]$ while $f^{-k+j}(x_k)<x_{j}$. In particular when $i=j-1$, the inequality $x_{j-1}\leq f^{-k+j-1}(x_k)$ implies that $f(x_{j-1})\leq f^{-k+j}(x_k)(<x_j)$. Thus 
		$$d^c(f(x_{j-1}), f^{-k+j}(x_k))\leq d^c(f(x_{j-1}), x_j)<\varepsilon.$$
		Consider the new collection
		$$\big\{x_0,x_1,\cdots,x_{j-1}, f^{-k+j}(x_k),f^{-k+j+1}(x_k), \cdots, f^{-1}(x_k),x_k,x_{k+1}, \cdots, x_{n-1},x_n\big\}.$$ 
		This is an $\varepsilon$-center pseudo orbit and  on $\mathcal{F}^c_{x_k}$, it satisfies that 
		$$f^k(x_0)<f(f^{-1}(x_k))=x_k\leq x_k.$$
		So we denote the new collection by $\Gamma''$ in Sub-case (b).
		\medskip
		
		Note that in Sub-case (b) the new defined $\varepsilon$-center pseudo orbit $\Gamma''$ satisfies that 
		$k(\Gamma'')\leq k(\Gamma)-1<k(\Gamma)$. Moreover, the two endpoints $x_0$ and $x_n$ are always kept.	Since the length of pseudo orbit is finite, this process ends in finite steps. The final $\varepsilon$-center pseudo orbit, denoted also by $\Gamma'$ as in Sub-case (a), satisfies Claim~\ref{Claim:keep-order-noncompact}.	
	\end{proof}
	
	\paragraph{Compact case} Assume each $\mathcal{F}^c_{x_i}, i=0,1,\cdots,n$ is compact and thus  $\theta_i\colon \mathbb{R}_i\rightarrow\mathcal{F}^c_{x_i}$ is a universal covering map. By the locally isometric property of $\theta_i$, one can see that the piece $\widehat{\Gamma}=\{0_0,0_1,0_2,\cdots,0_{n-1},0_n\}$ is an $\varepsilon$-pseudo orbit of the lifting dynamics $\widehat{f}$ where $0_i\in \mathbb{R}_i$ is the zero point on $\mathbb{R}_i$ for any $i=0,1,\cdots,n$.
	Note that $\widehat{f}(0_i)=t_{i+1}$. Applying the same arguments as in the non-compact case, one can assume that all $t_i,1\geq n$ have the same sign. This completes the proof of
	Lemma~\ref{Lem:ordered-center-chain}.
\end{proof}

\begin{remark}\label{Rem:ordered-chain}
	Since each $\widehat f|_{\mathbb{R}_i}$ is strictly increasing, thus when $\Gamma$ is taken to satisfy that  $-\varepsilon<t_i\leq 0$ for all $1\leq i\leq n$ as in Lemma~\ref{Lem:ordered-center-chain}, then $\widehat f^i(0_0)\leq 0_i$ for all $1\leq i\leq n$. The other case is symmetric.
\end{remark}

\textbf{One fixes the constants $0<\lambda<1$ and  $0<\eta_0<1/10^{-3}$  from Section~\ref{Section:setting} and $0<\varepsilon_0<1$ in Section~\ref{Section:lifting-dynamics}.}

\subsection{Lifting bundle}\label{Seciton:lifting-bundle}
Recall that we have fixed two $C^\infty$-sub-bundles $F^s$ and $F^u$ that are $C^0$-close to $E^s$ and $E^u$ respectively. In this section, for a bi-infinite center pseudo orbit $\Gamma$, we pull back the  bundle $F^s\oplus F^u$ over $M$ to $\bigsqcup\limits_{i\in\mathbb{Z}}\mathbb{R}_i$ through the  $C^1$-immersion map $\theta\colon \bigsqcup\limits_{i\in\mathbb{Z}}\mathbb{R}_i\rightarrow M$. Then we obtain a tubular neighborhood of $\bigsqcup\limits_{i\in\mathbb{Z}}\mathbb{R}_i$ on which one could define an exponential map to $M$ and a bundle map which is partially hyperbolic. The construction follows essentially from~\cite[Section 4.1]{GS-closing} where the ideas originates from~\cite{hps}. We sketch the arguments and conclusions here. Also some of the notations are borrowed from~\cite{GS-closing}.

Assume $\varepsilon\in(0,\varepsilon_0)$ and let $\Gamma=\big\{x_0,x_1,\cdots,x_n\big\}$ be an $\varepsilon$-center pseudo orbit 
from Section~\ref{Section:lifting-dynamics} that satisfies Lemma~\ref{Lem:ordered-center-chain}. 
Let $\displaystyle\Gamma_\infty=\big\{x_i\big\}_{i=-\infty}^{+\infty}$ be the extended bi-infinite center pseudo orbit such that $x_i=f^i(x_0)$ when $i\leq -1$ and $x_i=f^{i-n}(x_n)$ when $i\geq n+1$  as in Section~\ref{Section:lifting-dynamics}.
One pulls back the $C^\infty$-bundle $F^s\oplus F^u$ over $M$ by the $C^1$-leaf immersion $\theta\colon \bigsqcup\limits_{i\in\mathbb{Z}}\mathbb{R}_i\rightarrow M$ to get a fiber bundle $\theta^*(F^s\oplus F^u)$ over 
$\bigsqcup\limits_{i\in\mathbb{Z}}\mathbb{R}_i$. Thus one has the following commuting diagram.
\begin{displaymath}
\xymatrix{   \theta^*(F^s\oplus F^u) \ar[r]^{\theta_*} \ar[d]_{\pi} & F^s\oplus F^u \ar[d]^{\pi}\\
	\bigsqcup\limits_{i\in\mathbb{Z}}\mathbb{R}_i  \ar[r]^\theta& M}
\end{displaymath}

For each $t\in\mathbb{R}_i$ where $i\in\mathbb{Z}$, one pulls back the Riemannian metric on $F^s_{\theta_i(t)}\oplus F^u_{\theta_i(t)}$ and thus the fiber 
$$\theta^*(F^s\oplus F^u)_t \colon=\theta^*(F^s_{\theta_i(t)}\oplus F^u_{\theta_i(t)})$$
over $t\in\mathbb{R}_i$ is a linear space equipped with an inner product structure.
In this way one defines a metric $\|\cdot\|_t$ on $\theta^*(F^s\oplus F^u)_t$.
For any constant $\delta>0$, any $t\in\mathbb{R}_i$ and any interval $I\subset \mathbb{R}_i$, one denotes
$$\theta^*(F^s\oplus F^u)_t(\delta)=\big\{v\in \theta^*(F^s\oplus F^u)_t: \|v\|_t \leq \delta\big\},$$
$$\theta^*(F^s\oplus F^u)_I(\delta)=\bigcup_{t\in I}\theta^*(F^s\oplus F^u)_t(\delta),$$
and
$$\theta^*(F^s\oplus F^u)(\delta)=\bigcup_{t\in \bigsqcup_{i\in\mathbb{Z}}\mathbb{R}_i}\theta^*(F^s\oplus F^u)_t(\delta).$$
In particular, when $\delta=0$, the $0$-section $\theta^*(F^s\oplus F^u)(0)$ is reduced to $\bigsqcup\limits_{i\in\mathbb{Z}}\mathbb{R}_i$.

For $t\in\mathbb{R}_i$  and $v\in\theta^*(F^s\oplus F^u)_t$, let
$$\Phi(v)=\exp_{\theta_i(t)}({\theta}_*(v)).$$
In this way one defines the exponential map $\Phi\colon\theta^*(F^s\oplus F^u)\rightarrow M$.
Then by~\cite[Proposition 4.5]{GS-closing}, there exists a constant $\delta_0>0$ that depends only on $f,F^s,F^u$ such that for any $0<\delta\leq\delta_0$ and for any $t\in\bigsqcup\limits_{i\in\mathbb{Z}}\mathbb{R}_i$, the set $\Phi(\theta^*(F^s\oplus F^u)_{(t-\delta,t+\delta)}(\delta))$ is open in $M$. Moreover, the exponential map $\Phi$ is a diffeomorphism from $\theta^*(F^s\oplus F^u)_{(t-\delta,t+\delta)}(\delta)$ to its image. As a consequence, there exists a constant $C_0>1$ that depends on $f$ such that for any $t\in\bigsqcup\limits_{i\in\mathbb{Z}}\mathbb{R}_i$ and any $v\in\theta^*(F^s\oplus F^u)_{t}(\delta_0/C_0)$, there exists a unique vector
$$\theta^*f(v)\colon=\Phi^{-1}\circ f\circ\Phi(v)\in \theta^*(F^s\oplus F^u)_{t}(\delta_0/2),$$
where $\Phi^{-1}$ is defined from $\Phi\big(\theta^*(F^s\oplus F^u)_{(t-\delta_0,t+\delta_0)}(\delta_0)\big)$ to $\theta^*(F^s\oplus F^u)_{(t-\delta_0,t+\delta_0)}(\delta_0)$.
That is to say, one has the following commuting diagram:
\begin{displaymath}
	\xymatrix{   \theta^*(F^s\oplus F^u)(\delta_0/C_0) \ar[r]^{\theta^*f} \ar[d]_{\Phi} & \theta^*(F^s\oplus F^u)(\delta_0/2) \ar[d]^{\Phi}\\
		M  \ar[r]^{f}& M}
\end{displaymath}
Moreover, the lifting map $\theta^*f$ is a diffeomorphism from $\theta^*(F^s\oplus F^u)(\delta_0/C_0)$ to its image. 
On the other hand, one pulls back the Riemannian metric on $M$ through $\Phi$ to obtain a local metric $\widetilde{d}(\cdot,\cdot)$ on $\theta^*(F^s\oplus F^u)(\delta_0)$. To be precise, for any two $v_1,v_2\in \theta^*(F^s\oplus F^u)_{(t-\delta_0,t+\delta_0)}(\delta_0)$ for some $t\in\bigsqcup\limits_{i\in\mathbb{Z}}\mathbb{R}_i$, one has $\widetilde{d}(v_1,v_2)=d(\Phi(v_1),\Phi(v_2))$. See similar statements in~\cite[Page 32]{GS-closing}.

\begin{remark}\label{Rem:restriction-to-0-section}
	When restricted to the $0$-section $\theta^*(F^s\oplus F^u)(0)$ (which is  $\bigsqcup\limits_{i\in\mathbb{Z}}\mathbb{R}_i$ as pointed out above), the map 
	$$\theta^*f|_{\theta^*(F^s\oplus F^u)(0)}\colon \theta^*(F^s\oplus F^u)(0)\rightarrow \theta^*(F^s\oplus F^u)(0)$$ 
	is reduced to $\widehat f\colon \bigsqcup\limits_{i\in\mathbb{Z}}\mathbb{R}_i\rightarrow \bigsqcup\limits_{i\in\mathbb{Z}}\mathbb{R}_i$. In this sense, the map $\theta^*f$ is in fact an extension of $\widehat f$ to a tubular neighborhood of $\bigsqcup\limits_{i\in\mathbb{Z}}\mathbb{R}_i$.
	Thus in the following, we identify the $0$-section $\theta^*(F^s\oplus F^u)(0)$ to  $\bigsqcup\limits_{i\in\mathbb{Z}}\mathbb{R}_i$ and associated the order on  $\bigsqcup\limits_{i\in\mathbb{Z}}\mathbb{R}_i$ to $\theta^*(F^s\oplus F^u)(0)$.
	One has
	\[\theta^*f|_{\theta^*(F^s\oplus F^u)(0)}=\widehat f\colon \bigsqcup\limits_{i\in\mathbb{Z}}\mathbb{R}_i\rightarrow \bigsqcup\limits_{i\in\mathbb{Z}}\mathbb{R}_i.\]
\end{remark}

\begin{remark}\label{Rem:periodic-to-general}
	We point out that~\cite[Proposition 4.5]{GS-closing} deals with periodic center leaves while here we consider general ones. However, the existence of $\delta_0$ and $C_0$ only depends on the uniform hyperbolicity of the two extreme bundles $E^s/E^u$ (or equivalently on the constant $0<\lambda<1$) and the bounded norm of $f$.
\end{remark}

The partially hyperbolic splitting $TM=E^s\oplus E^c\oplus E^u$ induces a partially hyperbolic structure for the bundle map $\theta^*f\colon\theta^*(F^s\oplus F^u)(\delta_0/C_0)\rightarrow \theta^*(F^s\oplus F^u)(\delta_0/2)$ over the diffeomorphism $\widehat{f}\colon \bigsqcup\limits_{i\in\mathbb{Z}}\mathbb{R}_i\rightarrow \bigsqcup\limits_{i\in\mathbb{Z}}\mathbb{R}_i$ with  partially hyperbolic splitting
\begin{align*}
	T_{\bigsqcup\limits_{i\in\mathbb{Z}}\mathbb{R}_i}\theta^*(F^s\oplus F^u)  & =  \theta^*(E^s)|_{\bigsqcup\limits_{i\in\mathbb{Z}}\mathbb{R}_i}\oplus T(\bigsqcup\limits_{i\in\mathbb{Z}}\mathbb{R}_i)\oplus \theta^*(E^u)|_{\bigsqcup\limits_{i\in\mathbb{Z}}\mathbb{R}_i}\\
	& = \left(\theta^*(E^s)\oplus \theta^*(E^c)\oplus \theta^*(E^u)\right)|_{\bigsqcup\limits_{i\in\mathbb{Z}}\mathbb{R}_i}.
\end{align*}
The main ingredient is that $\theta\colon \bigsqcup\limits_{i\in\mathbb{Z}}\mathbb{R}_i\rightarrow M$ is a normally hyperbolic immersion following~\cite[Page 69]{hps}.
The detailed proofs can be found in~\cite[Proposition 4.6]{GS-closing}.
\medskip

\textbf{One fixes the constants $\delta_0>0$ and $C_0>1$ from Section~\ref{Seciton:lifting-bundle}.}

\subsection{Lipschitz shadowing of invariant manifolds and leaf conjugancy}\label{Section:Lipschtz-shadowing}
Recall that $X\in\mathcal{X}^r(M)$ is a $C^r$ vector field which is positively transverse to $E^s\oplus E^u$ (and also transverse to $F^s\oplus F^u$) everywhere and $X_\tau$ is the time-$\tau$ map generated by $X$. In this section, we consider the  lifting bundle map of the composition  $X_\tau\circ f$ (which can be seen as $C^r$-small  perturbations of $f$ provided $\tau$ small). This allows us to have a Lipschitz shadowing property of invariant manifolds (Theorem~\ref{Thm:Lip-shadowing}) and a leaf conjugacy map (Section~\ref{Section:leaf-conjugacy}).

The transverse property implies that $X$ has no singularity. 
Denote by $f_\tau=X_\tau\circ f\colon M\rightarrow M$ for every $\tau\in\mathbb{R}$.
One lifts $X$ to be a vector field $\widetilde{X}$ on $\theta^*(F^s\oplus F^u)(\delta_0)$ as follows
$$\widetilde{X}(v)= D\Phi^{-1} (X(\Phi(v))),~~~~~~\text{for}~\forall v\in \theta^*(F^s\oplus F^u)(\delta_0).$$
The lifting vector field $\widetilde{X}$ is $C^0$-continuous on $\theta^*(F^s\oplus F^u)(\delta_0)$, while for $\tau$ small, the time $\tau$-map $\widetilde{X}_\tau$ generated by $\widetilde{X}$ is a $C^1$-diffeomorphism from $\theta^*(F^s\oplus F^u)(\delta_0/2)$ to its image in $\theta^*(F^s\oplus F^u)(\delta_0)$ and satisfies
$\widetilde{X}_\tau=  \Phi^{-1}\circ X_\tau \circ\Phi$. One has the following theorem which is stated in~\cite[Lemma 4.9, Theorem 4.2\&4.3]{GS-closing}, and whose proofs are originated from~\cite[Theorem 6.8]{hps} and based on~\cite[Theorem 6.1]{hps} which shows the existence and properties of local invariant manifolds. Recall that we have fixed the constant $0<\eta_0<10^{-3}$  from Section~\ref{Section:setting}.

\begin{theorem}\label{Thm:Lip-shadowing}
	There exist two constants $\tau_0>0$ and $L>1$ such that for any $\tau\in[-\tau_0,\tau_0]$, the following properties are satisfied.
	\begin{enumerate}
		\item\label{item:lift-diffeo-family} The family of $C^1$-diffeomorphisms 
		$$\widetilde{f}_\tau=\widetilde{X}_\tau\circ \theta^*f\colon \theta^*(F^s\oplus F^u)(\delta_0/C_0)\rightarrow \theta^*(F^s\oplus F^u)(\delta_0)$$
		is well defined and $\widetilde{f}_\tau\rightarrow\theta^*f$ in the $C^1$-topology as $\tau\rightarrow 0$.
		Moreover, one has 
		$$\widetilde{d}(v,\widetilde{X}_\tau(v))\leq L\cdot|\tau|, ~~~~~~\forall v\in \theta^*(F^s\oplus F^u)(\delta_0/C_0).$$
		\item\label{item:invariant-section} There exists an $\widetilde{f}_\tau$-invariant $C^1$-smooth section $\sigma_{\tau}\colon\bigsqcup\limits_{i\in\mathbb{Z}}\mathbb{R}_i\rightarrow \theta^*({F}^s\oplus {F}^u)(\delta_0)$ which is defined as
		$$\sigma_{\tau}\left(\bigsqcup\limits_{i\in\mathbb{Z}}\mathbb{R}_i\right)=
		\bigcap_{n\in\mathbb{Z}}\widetilde{f}^n_{\tau} \left(\theta^*({F}^s\oplus {F}^u)(\delta_0/C_0)\right),$$
		and which varies continuous with respect to $\tau$. 
		The two $C^1$-manifolds 
		$$W^{cs}_\tau=\bigcap_{n=0}^{+\infty}\widetilde{f}^{-n}_{\tau} \left(\theta^*({F}^s\oplus {F}^u)(\delta_0/C_0)\right)~~\text{~and~}~~W^{cu}_\tau=\bigcap_{n=0}^{+\infty}\widetilde{f}^{n}_{\tau} \left(\theta^*({F}^s\oplus {F}^u)(\delta_0/C_0)\right),$$
		 intersect transversely at $\sigma_{\tau}\big(\bigsqcup\limits_{i\in\mathbb{Z}}\mathbb{R}_i\big)=W^{cs}_\tau\cap W^{cu}_\tau$.
		 Moreover, the section $\sigma_\tau$ converges to $\sigma_0$ in the $C^1$-topology as $\tau\rightarrow 0$, where $\sigma_0(\bigsqcup\limits_{i\in\mathbb{Z}}\mathbb{R}_i)=\bigsqcup\limits_{i\in\mathbb{Z}}\mathbb{R}_i$ is the $0$-section.
		\item\label{item:invariant-manifold} For any $t\in\mathbb{R}_i, i\in\mathbb{Z}$ and any $C^1$ map $\varphi\colon E^s_{\theta_i(t)}(\delta_0)\oplus E^u_{\theta_i(t)}(\delta_0)\rightarrow E^c_{\theta_i(t)}(\delta_0)$ satisfying $\varphi(0)=0, \|\partial\varphi/\partial \alpha\|\leq \eta_0$ for $\alpha=s,u$, the $C^1$-submanifold
		$$D^{su}_{\theta_i(t)}=\exp_{\theta_i(t)}\left(\grap(\varphi)\right)=
		\exp_{\theta_i(t)}\left(\left\{v^{su}+\varphi(v^{su})\colon v^{su}\in E^s_{\theta_i(t)}(\delta_0)\oplus E^u_{\theta_i(t)}(\delta_0)\right\}\right)$$
		intersects $\Phi\left(\sigma_{\tau}((t-\delta_0,t+\delta_0))\right)$ at a unique point $q=q(t,\tau,\varphi)$. Moreover, let $v_q^{su}\in E^s_{\theta_i(t)}(\delta_0)\oplus E^u_{\theta_i(t)}(\delta_0)$ be such that $q=\exp_{\theta_i(t)}\left(v_q^{su}+\varphi(v_q^{su})\right)$, then 
		$$\|v_q^{su}\|+\|\varphi(v_q^{su})\|\leq L\cdot|\tau|,$$
		which in particular implies that $\sigma_{\tau}(\bigsqcup\limits_{i\in\mathbb{Z}}\mathbb{R}_i)\subset \theta^*(F^s\oplus F^u)(L\cdot|\tau|)$.
		
	\end{enumerate}
\end{theorem}

\begin{remark}\label{Rem:identification-at-0-section}
Note that the map $\widetilde{f}_\tau=\widetilde{X}_\tau\circ\theta^*f$ in Theorem~\ref{Thm:Lip-shadowing} is the lift of $f_\tau=X_\tau\circ f$. In particular, when $\tau=0$, one has $\widetilde{f}_0=\theta^*f$ and as a consequence (recall Remark~\ref{Rem:restriction-to-0-section}), when restricted to the $0$-section $\theta^*({F}^s\oplus {F}^u)(0)=\bigsqcup\limits_{i\in\mathbb{Z}}\mathbb{R}_i$, it satisfies:
$$\widetilde{f}_0|_{\bigsqcup\limits_{i\in\mathbb{Z}}\mathbb{R}_i} = \theta^*f|_{\bigsqcup\limits_{i\in\mathbb{Z}}\mathbb{R}_i}=\widehat f\colon \bigsqcup\limits_{i\in\mathbb{Z}}\mathbb{R}_i\rightarrow \bigsqcup\limits_{i\in\mathbb{Z}}\mathbb{R}_i.$$
\end{remark}
\medskip


\subsubsection{Leaf conjugacy}\label{Section:leaf-conjugacy}

Given a constant $\eta\in (0,\eta_0]$ and an $(s+u)$-dimensional $C^\infty$-sub-bundle $F$ of $TM$ that is $\eta$-$C^0$-close to  $E^s\oplus E^u$. For a point $z\in M$ and a constant $0<\delta\leq\delta_0$, let 
$$D^F_z(\delta)=\exp_z(F_z(\delta)).$$ 
Then following the arguments in~\cite[Section 4.2]{GS-closing}\footnote{Different with our setting here, \cite[Section 4.2]{GS-closing} deals with a bundle $E^s\oplus F$ where $E^s$ is the contracting sub-bundle in the partially hyperbolic splitting and $F$ is an $u$-dimensional $C^\infty$-sub-bundle that is $C^0$-close to the expanding sub-bundle $E^u$.}
, there exists a constant $\delta_1=\delta_1(F,\eta)\in(0,\delta_0]$ such that the following properties are satisfied.
	\begin{enumerate}
		\item\label{Item:map} For any point $z\in M$, the set $D^{F}_z(\delta_1)$ is a $C^r$-smooth local manifold whose tangent bundle is $2\eta$-close to $E^s\oplus E^u$. To be precise,
		$$\angle(T_wD^{F}_z(\delta_1),(E^s\oplus E^u)_w)<2\eta,~~\forall w\in D^{F}_z(\delta_1).$$
		As a consequence, there exists a $C^r$-map $\varphi^{su}_z\colon E^s(\delta_1/2)\oplus E^u(\delta_1/2)\rightarrow E^c(\delta_1/2)$ satisfying that
		\begin{itemize}
			\item $\grap(\varphi^{su}_z)=\exp^{-1}\left(D^{F}_z(\delta_1)\right)\cap T_zM(\delta_1/2)$;
			\item $\varphi^{su}_z(0_z^{su})=0_z^c$ and $\|\partial\varphi^{su}_z/\partial\alpha\|<2\eta$ for $\alpha=s,u$.
		\end{itemize}
		\item\label{Item:foliation} Given a $C^1$-center curve $\gamma^c$ with radius $\delta_1$ and centered at a point $z\in M$,  the set
		$$B^{F}_{z}(\gamma^c,\delta_1)=\bigcup_{w\in\gamma^c}  D^{F}_w(\delta_1)$$ 
		is a neighborhood of $z$ in $M$ and the family
		$$\mathcal{D}^{F}(\gamma^c)=\left\{D^{F}_w(\delta_1)\colon w\in\gamma^c \right\}$$
		forms a $C^0$-foliation of $B^{F}_{z}(\delta_1)$. Moreover, the foliation $\mathcal{D}^{F}(\gamma^c)$ is transverse to any $C^1$-curve $\gamma\in B^{F}_{z}(\delta_1)$ that is tangent to the $\eta$-cone filed of $E^c$.
		\item\label{Item:uniform-bound} By the uniform transversality and compactness, for any constant $\delta\in[0,\delta_1]$, there exists $\delta'\in [0,\delta_1]$ such that given a $C^1$-center curve $\gamma^c$ with radius $\delta_1$ and centered at a point $z\in M$, if $\gamma_1$ and $\gamma_2$ are two $C^1$-curves tangent to the $\eta_0$-cone field of $E^c$ in $B^{F}_{z}(\gamma^c,\delta_1)$ with endpoints $w_i,w_i',i=1,2$ that satisfy
		$$w_1,w_2\in D^{F}_{w_0}(\delta_1)~~\text{and}~~w_1',w_2'\in D^{F}_{w_0'}(\delta_1)~~\text{where}~~w_0,w_0'\in\gamma^c,$$
		then $\length(\gamma_1)\geq \delta$ implies $\length(\gamma_2)\geq \delta'$.
	\end{enumerate}

Denote by $\widetilde{D}^{F}_t(\delta)=\Phi^{-1}(D^{F}_{\theta(t)}(\delta))$ for $0<\delta<\delta_1$. Then $$\mathcal{D}^{F}(\delta_1/10)=\left\{\widetilde{D}^{F}_t(\delta)\cap \theta^*(F^s\oplus F^u)(\delta_1/10)\colon t\in\bigsqcup_{i\in\mathbb{Z}} \mathbb{R}_i\right\}$$
forms a $C^0$-foliation of $\theta^*(F^s\oplus F^u)(\delta_1/10)$ that is transverse to $(\Phi^{-1})^*(E^c)$ everywhere.
Following the idea of~\cite[Section 7]{hps}, one can define a leaf conjugacy as follows.
There exists a constant $\tau_1=\tau_1(\eta,F,\delta_1)\in(0,\tau_0]$ where $\tau_0$ is from Theorem~\ref{Thm:Lip-shadowing}, such that for every $\tau\in[-\tau_1,\tau_1]$, the leaf conjugacy 
$$h_\tau\colon \bigsqcup_{i\in\mathbb{Z}}\mathbb{R}_i\rightarrow \bigsqcup_{i\in\mathbb{Z}}\sigma_\tau(\mathbb{R}_i)$$
where
$$h_\tau(t)=\sigma_\tau(\mathbb{R}_i)\cap \widetilde{D}^{F}_t(\delta_1),~~\forall t\in\mathbb{R}_i,$$
is well defined. In particular, when $\tau=0$, the map $h_0|_{\mathbb{R}_i}$ is the identity map $\id|_{\mathbb{R}_i}$.
\medskip

\textbf{The  constant $\eta$ and the bundle $F$ (thus $\delta_1=\delta_1(\eta,F)$ and $\tau_1=\tau_1(\eta,F,\delta_1)$ ) will be determined in Proposition~\ref{Prop:center-perturbation} in the next section.}

\section{Proofs of main theorems}\label{Section:proof-main}
In this section, we prove Theorem~\ref{Thm:main} and Corollary~\ref{Thm:generic}.
Let $f\in\dph_1^r(M)$, $r\in\mathbb{N}_{\geq 2}\cup\{\infty\}$ and  $X\in\mathcal{X}^r(M)$ be a vector field that is positively transverse to $E^s\oplus E^u$. 
For the proof of Theorem~\ref{Thm:main}, all the settings and assumptions are taken from Section~\ref{Section:preparations}.

\subsection{A global perturbation along center direction}

One has the following proposition that realizes perturbations along center directions. The idea essentially follows from  arguments in~\cite[Proposition 4.17]{GS-closing} with the difference that ~\cite[Proposition 4.17]{GS-closing} deals with periodic leaves while here we consider general cases. We will include the proof for completeness. Recall Remark~\ref{Rem:restriction-to-0-section} and~\ref{Rem:identification-at-0-section},  the $0$-section $\theta^*({F}^s\oplus {F}^u)(0)$ is identified to $\bigsqcup\limits_{i\in\mathbb{Z}}\mathbb{R}_i$ and thus is associated with an order from $\bigsqcup\limits_{i\in\mathbb{Z}}\mathbb{R}_i$. Also restricted on $\theta^*({F}^s\oplus {F}^u)(0)$, one has $\widetilde{f}_0|_{\bigsqcup\limits_{i\in\mathbb{Z}}\mathbb{R}_i} = \theta^*f|_{\bigsqcup\limits_{i\in\mathbb{Z}}\mathbb{R}_i}=\widehat f$.
\begin{proposition}\label{Prop:center-perturbation}
	There exist a constant $\eta=\eta(f,X)\in (0,\eta_0]$, an $(s+u)$-dimensional $C^\infty$ sub-bundle $F\subset TM$ that is $\eta$-close to $E^s\oplus E^u$ and $\tau_2=\tau_2(f,X,\eta,F)\in (0,\tau_1]$ such that the following properties are satisfied.
	For $\tau\in[-\tau_2,\tau_2]$, let 
	$h_\tau\colon \bigsqcup\limits_{i\in\mathbb{Z}}\mathbb{R}_i\rightarrow \bigsqcup\limits_{i\in\mathbb{Z}}\sigma_\tau(\mathbb{R}_i)$
	be the leaf conjugacy
	$$h_\tau(t)=\sigma_\tau(\mathbb{R}_i)\cap \widetilde{D}^{F}_t(\delta_1),~~\forall t\in\mathbb{R}_i,$$
	then for any $\tau\in[-\tau_2,\tau_2]$ there exists $\Delta_\tau>0$ such that for any $t\in\mathbb{R}_i$, one has
	\begin{enumerate}
		\item if $0<\tau<\tau_2$, then 
		$$\widetilde{f}_{\tau}\circ h_\tau(t)>h_\tau(\theta^*f(t)+\Delta_\tau);$$
		\item if $-\tau_2<\tau<0$, then 
		$$\widetilde{f}_{\tau}\circ h_\tau(t)<h_\tau(\theta^*f(t)-\Delta_\tau);$$
	\end{enumerate}
	Here the order on each $\sigma_{\tau}(\mathbb{R}_i)$ is inherited from that on $\mathbb{R}_i$.
	The constants $\delta_1=\delta_1(\eta,F)$ and $\tau_1=\tau_1(\eta,F,\delta_1)$ are from Section~\ref{Section:leaf-conjugacy}.
\end{proposition}

\begin{proof}
	We only prove the case when $\tau>0$ since the other case when $\tau<0$ is symmetric.
	
	
\paragraph{The lifting vector field $\widehat X$ on $TM$}	
    One first lifts $X$ to $TM$ in the following way: given a point $z\in M$, let $\widehat X_z=D\exp^{-1}_z(X)$. Then $\widehat X_z$ is a vector field which is well defined in a neighborhood of the zero vector $0_z\in T_zM$. Note that for $v\in T_zM$, its tangent space $T_v(T_zM)\cong T_zM$, thus it admits a direct sum $T_v(T_zM)=E_z^s(v)\oplus E_z^c(v)\oplus E_z^u(v)$ induced by $T_zM=E^s_z\oplus E^c_z\oplus E^u_z$.
	For $v\in T_zM$ close to $0_z$, denote by 
	$$\widehat X_z(v)=\widehat X_z^s(v)\oplus \widehat X_z^c(v)\oplus \widehat X_z^u(v),$$
	where $\widehat X_z^\alpha(v)\in E_z^\alpha(v), \alpha=s,c,u$. 
	Recall that  $X$ is  positively transverse to $E^s\oplus E^u$,
	thus there exist constants $\delta_2>0$, $a_0>0$ and $b_0>0$ such that for any $z\in M$,
	\begin{itemize}
		\item for any $v\in T_zM(\delta_2)=E^s_z(\delta_2)\oplus E^c_z(\delta_2)\oplus E^u_z(\delta_2)$, it satisfies
		$$\|\widehat X_z^c(v)\|\geq b_0,~~\text{and}~~ \|\widehat X_z^s(v)+\widehat X_z^u(v)\|<a_0\|\widehat X_z^c(v)\|\leq 2 a_0\|\widehat X\|.$$
		\item for any $v\in E^s_z(\delta_2)\oplus E^u_z(\delta_2)$, the curve $\exp_z\big(v\times  E^c_z(\delta_2)\big)$ is tangent to the $\eta_0$-cone field of $E^c$.
	\end{itemize}
    Recall that one has defined an order on $E^c_z$ that is induced from the orientation of $E^c$ in Section~\ref{Section:setting}.
    One thus can associate the order of $E^c_z$ to $E^c_z(v)$ for each $v\in T_zM$. Then we have $\widehat X_z^c(v)\geq b_0>0$ for every vector $v\in T_zM(\delta_2)$.

\paragraph{The choices of constants $\eta,\tau'$ and the $C^\infty$ bundle $F$ }    
    Let $\tau_0>0$ and $L>1$ be the two constants given by Theorem~\ref{Thm:Lip-shadowing}. 
    Now we take the constant $\eta=\eta(f,X)\in (0,\eta_0]$ and the $(s+u)$-dimensional $C^\infty$ sub-bundle $F\subset TM$ as following.
    \begin{itemize}
    	\item Take $\eta$ to satisfy
    	$$0<\eta<\min\left\{\eta_0,\frac{b_0}{8\cdot \left[\big(\|Df\|+1\big)\cdot L+ 2 a_0\|\widehat X\| \right]}\right\}.$$
    	\item Take an $(s+u)$-dimensional $C^\infty$ sub-bundle $F\subset TM$ that is $\eta$-close to $E^s\oplus E^u$, i.e. $\angle(F,E^s\oplus E^u)<\eta$.
    \end{itemize}
    Associated to $\eta$ and $F$, there exist $\delta_1=\delta_1(\eta,F)\in (0,\delta_0]$ and $\tau_1=\tau_1(\eta,F,\delta_1)\in (0,\tau_0]$ that are from Section~\ref{Section:leaf-conjugacy}. Set $\delta'=10^{-2}\min\big\{\delta_1,\delta_2\big\}$ and take a constant $\tau'=\tau'(f,X,\eta,F)\in (0,\tau_0]$ that satisfies
    $$0<\tau'<\min\left\{\tau_1,\frac{\delta'}{100L\big(\|X\|+1\big)\|Df\|}\right\},$$
    where $\|X\|=\sup_{x\in M}\|X(x)\|$ and $\|Df\|=\sup_{x\in M}\|Df(x)\|$.
    
\paragraph{The estimation of moving along the center direction}
    Fix $t\in\mathbb{R}_i$ and denote by $z_0=\theta_i(t)$. 
    For any $\tau\in[-\tau',\tau']$, let $q_\tau=\Phi(h_\tau(t))$.
    By the choices of constants and Theorem~\ref{Thm:Lip-shadowing}, one has that $q_\tau\in D^F_{z_0}(\delta')$. 
    Moreover, considering the direct sum $\exp^{-1}(q_\tau)=v_{q_\tau}^{su}+v_{q_\tau}^c\in E^{su}_{z_0}\oplus E^c_{z_0}$, it satisfies that 
    $$\|v_{q_\tau}^{su}\|\leq L\cdot\tau\leq \frac{\delta'}{100\big(\|X\|+1\big)\|Df\|}.$$
    Let $\varphi^{su}_{z_0}\colon E^s_{z_0}(\delta_1/2)\oplus E^u_{z_0}(\delta_1/2)\rightarrow E^c_{z_0}(\delta_1/2)$ be the $C^r$-map satisfying that 
    $$\grap\big(\varphi^{su}_{z_0}\big)=\exp^{-1}\left(D^{F}_{z_0}(\delta_1)\right)\cap T_zM(\delta_1/2),$$ 
    then  one has 
    $$v_{q_\tau}^c=\varphi^{su}_{z_0}\big(v_{q_\tau}^{su}\big)<2\eta\cdot\|v_{q_\tau}^{su}\|\leq 2\eta\cdot L\cdot\tau\leq 2\eta\cdot\frac{\delta'}{100\big(\|X\|+1\big)\|Df\|}.$$
     
    Note that $z_0=\theta_i(t)$ and $f(z_0)=\theta_{i+1}(\theta^*f(t))$. By the choices of constants above,
    for any $\tau\in[-\tau',\tau']$, one has
    $$f(q_\tau)\in  B_{f(z_0)}(\delta'/10) ~\qquad~\text{and}~\qquad~ f_\tau(q_\tau)=X_\tau\circ f(q_\tau)\in B_{f(z_0)}(\delta').$$
    Consider the direct sum 
    $$\exp^{-1}_{f(z_0)}(f(q_\tau))=v_{f(q_\tau)}^{su}+v_{f(q_\tau)}^c\in E^{su}_{f(z_0)}\oplus E^c_{f(z_0)},$$
    one has:
    \begin{claim}\label{Claim:center-estimation}
    	There exist a continuous function $\xi\colon \mathbb{R}\rightarrow \mathbb{R}^+$ that depends on $f$ and $X$ satisfying that
    	\begin{enumerate}
    		\item\label{Item:infinitesimal} $\xi(\tau)={\rm o}(\tau)$ as $\tau\rightarrow 0$;
    		\item\label{Item:estimation} the following estimates holds:
    		\begin{itemize}
    			\item $\|v_{f(q_\tau)}^{su}\|\leq\left(\|Df(z_0)\|+\xi(\tau)\right)\cdot\|v_{q_\tau}^{su}\|\leq \left(\|Df(z_0)\|+\xi(\tau)\right)\cdot L\tau$;
    			\item $\|v_{f(q_\tau)}^{c}\|\leq \left(\|Df|_{E^c_{z_0}}\|+\xi(\tau)\right)\cdot\|v_{q_\tau}^{c}\|\leq 2\eta\cdot \left(\|Df|_{E^c_{z_0}}\|+\xi(\tau)\right)\cdot L\tau.$
    		\end{itemize}
    	The second formula implies that 
    	$$v_{f(q_\tau)}^{c}\geq-2\eta\cdot\left(\|Df|_{E^c_{z_0}}\|+\xi(\tau)\right)\cdot L\tau.$$
    	\end{enumerate}
    \end{claim}
    \begin{proof}[Proof of the Claim]
    	The proof is essentially because the map $\exp^{-1}_{f(z_0)}\circ f\circ \exp_{z_0}$ is $C^1$-smooth at $0_{z_0}$  and $D\big(\exp^{-1}_{f(z_0)}\circ f\circ \exp_{z_0}\big)(0_{z_0})=Df(z_0)$. See the details in the proof of~\cite[Claim 4.19]{GS-closing}.
    \end{proof}

\paragraph{The action of $\widehat X_\tau$ on $\exp^{-1}_{f(z_0)}(f(q_\tau))$}
Note first that 
$$\exp_{f(z_0)}\circ\widehat X_\tau\big( \exp^{-1}_{f(z_0)}(f(q_\tau))\big)=X_\tau\circ f(q_\tau)=f_\tau(q_\tau)=\Phi\circ \widetilde f_\tau(h_\tau(t))\colon= p_\tau.$$
This gives 
$$\widehat X_\tau\big( \exp^{-1}_{f(z_0)}(f(q_\tau))\big)=\exp_{f(z_0)}^{-1}\big(\Phi\circ \widetilde f_\tau(h_\tau(t))\big)=\exp_{f(z_0)}^{-1}(p_\tau).$$
We consider the direct sum
$$\widehat X_\tau\big( v_{f(q_\tau)}^{su}+v_{f(q_\tau)}^c\big)=\widehat X_\tau\big( \exp^{-1}_{f(z_0)}(f(q_\tau))\big)=v_{p_\tau}^{su}+v_{p_\tau}^c\in E^{su}_{f(z_0)}\oplus E^c_{f(z_0)}.$$
It satisfies for every $\tau\in[-\tau',\tau']$ that 
\begin{itemize}
	\item $\displaystyle\|v_{p_\tau}^{su}\|\leq  \|v_{f(q_\tau)}^{su}\|+\sup_{v\in T_{f(z_0)}M(\delta_2)}\left(\|\widehat X_{f(z_0)}^s(v)+\widehat X_{f(z_0)}^u(v)\|\right)\cdot\tau\leq \|v_{f(q_\tau)}^{su}\|+2 a_0 \|\widehat X\|\cdot \tau.$
	
    \item $\displaystyle v_{p_\tau}^c\geq -\|v_{f(q_\tau)}^c\|+\inf_{v\in T_{f(z_0)}M(\delta_2)}\|\widehat X_{f(z_0)}^c(v)\| \cdot \tau\geq -\|v_{f(q_\tau)}^c\|+b_0\cdot\tau.$
\end{itemize}
Combined with Claim~\ref{Claim:center-estimation}, one has
$$\|v_{p_\tau}^{su}\|\leq \left[\big(\|Df(z_0)\|+\xi(\tau)\big)\cdot L+ 2 a_0 \|\widehat X\|)\right]\cdot\tau ~~~\text{and}~~~
v_{p_\tau}^c\geq b_0\cdot \tau -2\eta\cdot\left(\|Df|_{E^c_{z_0}}\|+\xi(\tau)\right)\cdot L\cdot\tau.$$

\paragraph{The constant $\tau_2$} By Item~\ref{Item:infinitesimal} of Claim~\ref{Claim:center-estimation}, since $\xi(\tau)$ is a positive infinitesimal as $\tau\rightarrow 0$, there exists $\tau_2\in(0,\tau']$ such that for any $\tau\in[-\tau_2,\tau_2]$, the following holds
$$0<\xi(\tau)<1 ~~\text{and}~~\left[\big(\|Df\|+1\big)\cdot L+ 2 a_0 \|\widehat X\|)\right]\cdot|\tau|<\delta'.$$
Consider the disk $D^F_{f(z_0)}(\delta')$ and let Let $\varphi^{su}_{f(z_0)}\colon E^s_{f(z_0)}(\delta_1/2)\oplus E_{f(z_0)}^u(\delta_1/2)\rightarrow E_{f(z_0)}^c(\delta_1/2)$ be the $C^r$-map whose graph determines $D^F_{f(z_0)}(\delta')$ through the exponential map $\exp_{f(z_0)}$. When $\tau\in(0,\tau_2]$, by the choice of $\eta$, one has
\begin{align*}
	v_{p_\tau}^c- \varphi^{su}_{f(z_0)}(v_{p_\tau}^{su})
	&>b_0\cdot \tau -2\eta\cdot\left(\|Df|_{E^c_{z_0}}\|+1\right)\cdot L\cdot\tau- 2\eta\|v_{p_\tau}^{su}\|\\
	&> b_0\cdot \tau - 2\eta\cdot \left[2\big(\|Df\|+1\big)\cdot L+ 2 a_0\|\widehat X\| \right]\cdot\tau\\
	&>b_0\cdot \tau - \frac{b_0}{2}\cdot \tau =\frac{b_0}{2}\cdot \tau
\end{align*}
Thus when $0<\tau<\tau_2$, the point 
$$p_\tau=\exp_{f(z_0)}(v_{p_\tau}^{su}+v_{p_\tau}^c)=f_\tau(q_\tau)=\Phi\circ \widetilde f_\tau(h_\tau(t))$$
belong to $D^{F}_{\theta_{i+1}(t)}(\delta')$ for some $t'\in\mathbb{R}_{i+1}$ that satisfies
$$t'>\theta^*f(t).$$

Finally, consider the segment 
$$\gamma\colon=\exp_{f(z_0)}\left(\big\{v_{p_\tau}^{su}+v^c\colon  \varphi^{su}_{f(z_0)}(v_{p_\tau}^{su}) \leq v^c\leq  v_{p_\tau}^c \big\}\right).$$
One has that $\gamma$ is tangent to the $\eta_0$-cone field of $E^c$ with length uniformly bounded from from below (depending on $\tau$). By Item~\ref{Item:uniform-bound} of Section~\ref{Section:leaf-conjugacy}, there exists $\Delta_\tau>0$ such that $t'-\theta^*f(t)>\Delta_\tau$.
As a consequence, one has
$$\widetilde{f}_{\tau}\circ h_\tau(t)>h_\tau(\theta^*f(t)+\Delta_\tau).$$
This completes the proof.
\end{proof}

\begin{remark}
	The constant $\Delta_\tau$ depends on $f,X,F$ and $\tau$ but is independent with the choice of center pseudo orbit $\Gamma$.
\end{remark}

\subsection{Proof of Theorem~\ref{Thm:main}}


Based on Proposition~\ref{Prop:center-perturbation}, we are ready to prove Theorem~\ref{Thm:main}.
As we have mentioned before, we have to deal with general orbits of center leaves rather than periodic ones which are considered in the proof of~\cite[Theorem 2.3]{GS-closing}.

\begin{proof}[Proof of Theorem~\ref{Thm:main}]
	Let $f\in\dph_1^r(M), r\in\mathbb{N}_{\geq 2}$. 
	Assume $X\in\mathcal{X}^r(M)$ is a vector field that is transverse to $E^s\oplus E^u$ and the constant $\varepsilon_0$ is taken as in Section~\ref{Section:setting}.
	Let $L>1$ be from Theorem~\ref{Thm:Lip-shadowing}.
	Let the constant $\eta>0$, the $C^\infty$-bundle $F\subset TM$ and the constant $\tau_2$ be from Proposition~\ref{Prop:center-perturbation}.
	For each 
	$$k>2\max\left\{\frac{1}{\tau_2},\frac{L+1}{\varepsilon_0}\right\},$$
	let $\Delta_{1/k}>0$ be the constant associated to $\tau=1/k$ from Proposition~\ref{Prop:center-perturbation} and take $\varepsilon\in (0,\varepsilon_0)$ satisfying
	$$0<\varepsilon<\min\big\{1/k,\Delta_{1/k}\big\}.$$
	\medskip
	
	Assume $x\dashv_f y$. By Proposition~\ref{Pro:center-pseudo-orbit}, one takes an $\varepsilon$-center pseudo orbit
	$$\Gamma=\big\{x_0,x_1,\cdots,x_n\big\}$$ 
	such that
	 $$d(x_0,x)<\varepsilon ~~\text{and}~~d(x_n,y)<\varepsilon.$$ 
	Moreover, one assumes that $\Gamma$ satisfies Lemma~\ref{Lem:ordered-center-chain}. 
	And if $x\in\CR(f)$, i.e. $y=x$, one can take $\Gamma$ to satisfy that $x_0=x_n$ by
	Proposition~\ref{Pro:center-pseudo-orbit} and Remark~\ref{Rem:center-pseudo-orbit}. 
	Without loss of generality, we assume that $-\varepsilon<t_i\leq 0$ for each $0\leq i\leq n$ and the other case is symmetric.
	
	Let  $\displaystyle\Gamma_\infty=\big\{x_i\big\}_{i=-\infty}^{+\infty}$ be the extended bi-infinite center pseudo orbit such that $x_i=f^i(x_0)$ when $i\leq -1$ and $x_i=f^{i-n}(x_n)$ when $i\geq n+1$.
	Denote 
	$$
	\theta^*f:~\theta^*(F^s\oplus F^u)(\delta_0/C_0) \to \theta^*(F^s\oplus F^u)(\delta_0/2)
	$$
	the bundle dynamics associated to $\Gamma_\infty$ as Section \ref{Seciton:lifting-bundle}. Moreover, let
	$$
	\widetilde{f}_\tau=\widetilde{X}_\tau\circ \theta^*f\colon \theta^*(F^s\oplus F^u)(\delta_0/C_0)\rightarrow \theta^*(F^s\oplus F^u)(\delta_0)
	$$
	be the action of lifting vector field of $X$ composed with $\theta^*f$ as Section \ref{Section:Lipschtz-shadowing}.
	
	From the choice of center pseudo orbit $\Gamma$, one knows that
	$$\widetilde{f}_0^i(0_0)=(\theta^*f)^i(0_0)\leq 0_i, ~~\forall 1\leq i\leq n.$$
	In the trivial case when $\widetilde f_0^n(0_0)=0_n$ which is equivalent to that $f^n(x_0)=x_n$, the collection $\big\{x_0,x_1,\cdots,x_n\big\}$ is  a piece of orbit segment from $x_0$ to $x_n$ and the conclusion holds automatically.
	
	We now assume that $\widetilde f_0^n(0_0)<0_n$. Consider the bundle dynamics $\theta^*f$ on the lifting bundle $\theta^*(F^s\oplus F^u)$ and its perturbation $\widetilde f_{1/k}=\widetilde{X}_{1/k}\circ \theta^*f$ defined on $\theta^*(F^s\oplus F^u)(\delta_0/C_0)$. Let $h_{1/k}\colon \bigsqcup\limits_{i\in\mathbb{Z}}\mathbb{R}_i\rightarrow \bigsqcup\limits_{i\in\mathbb{Z}}\sigma_{1/k}(\mathbb{R}_i)$ be the leaf conjugacy in Section~\ref{Section:leaf-conjugacy}. Note that by the choices of constants, one has 
	$$-\Delta_{1/k}<-\varepsilon<t_i\leq 0,~~\forall 1\leq i\leq n.$$	
	By Proposition~\ref{Prop:center-perturbation}, one has the following estimation for the zero point $0_0\in\mathbb{R}_0$:
	\begin{align*}
		\widetilde{f}_{1/k}^n\left(h_{1/k}\big(0_0\big)\right)
		&> \widetilde{f}_{1/k}^{n-1}\left(h_{1/k}\big(\theta^*f(0_0)+\Delta_{1/k}\big)\right)
		= \widetilde{f}_{1/k}^{n-1}\left(h_{1/k}\big(t_1+\Delta_{1/k}\big)\right)\\
		&> \widetilde{f}_{1/k}^{n-1}\left(h_{1/k}\big(0_1\big)\right)\\
		&>\cdots\cdots \\
		&>  \widetilde{f}_{1/k}\left(h_{1/k}\big(0_{n-1}\big)\right)\\
		&> h_{1/k}\big(\theta^*f(0_{n-1})+\Delta_{1/k}\big)= h_{1/k}\big(t_n+\Delta_{1/k}\big)\\
		&> h_{1/k}\big(0_n\big).
	\end{align*}

    Recall that $\widetilde f_0^n(0_0)<0_n$ and $\sigma_{\tau}$ as well as $h_\tau$ varies continuous with respect to $\tau$. Thus there exists $\tau_k\in (0,1/k)$ such that 
    $$\widetilde{f}_{\tau_k}^n\left(h_{\tau_k}\big(0_0\big)\right)=h_{\tau_k}\big(0_n\big).$$
    Let $p_k=\Phi\big(h_{\tau_k}\big(0_0\big)\big)$ and $q_k=\Phi\big(h_{\tau_k}\big(0_n\big)\big)$, then it satisfies
    $$f_{\tau_k}^n\big(p_k\big)=\big(X_{\tau_k}\circ f\big)^n\big(p_k\big)=q_k.$$
    Moreover, by Theorem~\ref{Thm:Lip-shadowing}, one has that 
    $$d(x,p_k)\leq d(x,x_0)+d(x_0,p_k)<\varepsilon+ L\cdot\tau_k<\frac{L+1}{k},$$
    $$d(y,q_k)\leq d(y,x_n)+d(x_n,q_k)<\varepsilon+ L\cdot\tau_k<\frac{L+1}{k}.$$

    In particular, when $x\in\CR(f)$, i.e. $y=x$, one has $x_0=x_n$ as well as $0_0=0_n$, which implies that
    $$p_k=\Phi\big(h_{\tau_k}\big(0_0\big)\big)=\Phi\big(h_{\tau_k}\big(0_n\big)\big)=q_k.$$
    As a consequence, one has that $f_{\tau_k}^n\big(p_k\big)=\big(X_{\tau_k}\circ f\big)^n\big(p_k\big)=p_k.$
    
    By taking $k\rightarrow\infty$, this allows $\varepsilon$ to be arbitrarily small for the $\varepsilon$-center pseudo orbit $\Gamma$. This finishes the proof of Theorem~\ref{Thm:main}.
\end{proof}

\subsection{Proof of  Corollary~\ref{Thm:generic}}

Based on Theorem~\ref{Thm:main}, the proof of Item~\ref{Item:equivalent-of-attainability} in Corollary~\ref{Thm:generic} is obtained from the classical Baire arguments, see for instance~\cite[Section 5]{BC-chain-connecting} and~\cite[Section 3]{gw}. 
Item~\ref{Item:dense-periodic orbit} \&~\ref{Item:global-transitive}  are direct consequences of Item~\ref{Item:equivalent-of-attainability}.
We provide a short proof for completeness of the paper.
As have mentioned,  the reason that we consider the set $\deph^r_1(M)$ 
in Corollary~\ref{Thm:generic} is because $\deph^r_1(M)$ forms an open set in $\ph^r(M)$.


\begin{proof}[Proof of  Corollary~\ref{Thm:generic}]
	As a consequence of~\cite[Theorem A]{GS-closing}, there exists a dense $G_\delta$ subset $\mathcal{R}_0$ in $\deph_1^r(M)$ such that every $f\in\mathcal{R}_0$ satisfies $\Omega(f)=\overline{{\rm Per}(f)}$.
	
	Take a countable basis $\big\{O_n\big\}_{n\in\mathbb{N}}$ of $M$ and let $\big\{W_n\big\}_{n\in\mathbb{N}}$ be the countable collection where every $W_n$ is the union of finitely many elements in $\big\{O_n\big\}_{n\in\mathbb{N}}$. For each $n,m\in\mathbb{N}$ one defines the following two open sets $\mathcal{H}_{n,m}$ and $\mathcal{N}_{n,m}$ in $\deph_1^r(M)$:
	\begin{itemize}
		\item $f\in \mathcal{H}_{n,m}$ if there exists a neighborhood $\mathcal{U}$ of $f$ in $\deph_1^r(M)$ such that for any $g\in\mathcal{U}$, there exists $k\geq 1$ satisfying $g^k(W_n)\cap W_m\neq\emptyset$;
		\item $f\in \mathcal{N}_{n,m}$ if there neighborhood $\mathcal{U}$ of $f$ in $\deph_1^r(M)$ such that for any $g\in\mathcal{U}$ and for any $k\geq 1$, it satisfies $g^k(W_n)\cap W_m=\emptyset$.
	\end{itemize}
	It is not difficult to verify that $\mathcal{N}_{n,m}={\rm Int} \big(\deph_1^r(M)\setminus  \mathcal{H}_{n,m}\big)$ where ${\rm Int}(\cdot)$ means taking the interior. Thus $\mathcal{H}_{n,m}\cup \mathcal{N}_{n,m}$ is open dense in $\deph_1^r(M)$. 
	Let
	$$\mathcal{R}=\mathcal{R}_0\cap \left(\bigcap_{n,m\in\mathbb{N}}\big(\mathcal{H}_{n,m}\cup \mathcal{N}_{n,m}\big)\right).$$
	Then $\mathcal{R}$ is a dense $G_\delta$ subset in $\deph_1^r(M)$.
	We claim that the conclusions of Corollary~\ref{Thm:generic} hold for $f\in\mathcal{R}$.
 
    	Let $f\in\mathcal{R}$.    	
    	Assume $x\dashv_f y$. For any neighborhood $U$ of $x$ and any neighborhood $V$ of $y$, take $n,m\in\mathbb{N}$ such that $W_n\subset U$ and $W_m\subset V$. By Theorem~\ref{Thm:main}, for any neighborhood $\mathcal{U}$ of $f$ in $\deph_1^r(M)$, there exists $g\in\mathcal{U}$ and $k\geq 1$ such that $g^k(W_n)\cap W_m\neq\emptyset$. This implies that $f\in\mathcal{H}_{n,m}$ since $f\in \mathcal{H}_{n,m}\cup \mathcal{N}_{n,m}$. As a consequence, there exists $\ell>1$ such that $f^\ell(U)\cap V\neq\emptyset$ by the definition of $\mathcal{H}_{n,m}$ together with the fact $W_n\subset U$ and $W_m\subset V$. Since $U$ and $V$ are taken arbitrarily, one has $x\prec_f y$. This proves Item~\ref{Item:equivalent-of-attainability}.
    	
    	By Item~\ref{Item:equivalent-of-attainability}, for any $x\in\CR(f)$, it satisfies $x\prec_f x$ which is equivalent to say $x\in\Omega(f)$. Thus one has $\CR(f)=\Omega(f)$. Since $f\in\mathcal{R}$, one has $\Omega(f)=\overline{{\rm Per}(f)}$, This verifies Item~\ref{Item:dense-periodic orbit}.
    	
    	To prove Item~\ref{Item:global-transitive}, assume $f\in\mathcal{R}$ is chain-transitive. For any two non-empty open sets $U,V\subset M$, there exists $x\in U$ and $y\in V$ such that $x\dashv_f y$. By Item~\ref{Item:equivalent-of-attainability}, one has $x\prec_f y$, thus there exists $n\geq 1$ such that $f^n(U)\cap V\neq\emptyset$. This proves that $f$ is transitive.
\end{proof}	

\bibliographystyle{plain}

\flushleft{\bf Yi Shi} \\
School of Mathematics, Sichuan University, Chengdu, 610065, P. R. China\\
\textit{E-mail:} \texttt{shiyi@scu.edu.cn}\\

\flushleft{\bf Xiaodong Wang} \\
School of Mathematical Sciences,  CMA-Shanghai, Shanghai Jiao Tong University, Shanghai, 200240, P. R. China\\
\textit{E-mail:} \texttt{xdwang1987@sjtu.edu.cn}\\

\end{document}